\documentclass[12pt]{amsart}
\usepackage[toc,page]{appendix}
\usepackage{amssymb, latexsym, euscript}
\usepackage{amssymb}
\usepackage{latexsym}
\usepackage{amsmath}

      \def\sC{{\mathfrak C}}
\def\sD{{\mathfrak D}}      
   \def\sH{{\mathfrak H}}   
      \def\sL{{\mathfrak L}}
\def\sM{{\mathfrak M}}   \def\sN{{\mathfrak N}}

      \def\dC{{\mathbb C}}
\def\dD{{\mathbb D}}

   \def\dN{{\mathbb N}}   
      \def\dR{{\mathbb R}}

   \def\cB{{\mathcal B}}   
\def\cD{{\mathcal D}}      
   \def\cH{{\mathcal H}}   
      \def\cL{{\mathcal L}}
\def\cM{{\mathcal M}}   \def\cN{{\mathcal N}}

   \def\bB{{\mathbf B}}

\def\cRS{\mathcal{RS}}

\def\half{{\frac{1}{2}}}

\def\ran{{\rm ran\,}}
\def\cran{{\rm \overline{ran}\,}}
\def\dom{{\rm dom\,}}
\def\mul{{\rm mul\,}}
\def\cdom{{\rm \overline{dom}\,}}
\def\clos{{\rm clos\,}}
\def\dim{{\rm dim\,}}
\def\codim{{\rm codim\,}}

\def\uphar{{\upharpoonright\,}}

\def\f{\varphi}

\newtheorem{theorem}{Theorem}[section]
\newtheorem{lemma}[theorem]{Lemma}
\newtheorem{proposition}[theorem]{Proposition}
\newtheorem{corollary}[theorem]{Corollary}
\newtheorem{definition}[theorem]{Definition}
\newtheorem{remark}[theorem]{Remark}

\numberwithin{equation}{section}

\def\wt{\widetilde}

\begin{document}
\title[Shorting, parallel addition and form sums]
{Shorting, parallel addition and form sums of nonnegative selfadjoint linear relations}
\author[Yu.M.~Arlinski\u{\i}]{Yu.M.~Arlinski\u{\i}}
\address{Volodymyr Dahl East Ukrainian National University \\
pr. Central 59-A, Severodonetsk, 93400, Ukraine}
\email{yury.arlinskii@gmail.com}



\vskip 1truecm
\thispagestyle{empty}
\baselineskip=12pt


\begin{abstract}
We extend the operations of shorting and parallel addition from the cone of bounded nonnegative selfadjoint operators in a Hilbert space to the set of all nonnegative selfadjoint linear relations. New properties of these operations and connections with the Cayley transforms are established. It is shown that for a pair of nonnegative selfadjoint linear relations there exist, in general, two arithmetic--harmonic means. Applications of the arithmetic, harmonic, and arithmetic--harmonic means to the theory of nonnegative selfadjoint extensions of nonnegative symmetric linear relations are given.
\end{abstract}
\maketitle


\section{Introduction}
We will use the following notations: $\dom A$, $\ran A$, and $\ker A$ are the domain, the range, and the kernel of a linear operator/linear relation $A$, $\cran A$ and $\clos{\cL}$ denote the closure of $\ran A$ and of the set $\cL$, respectively. By ${\rm s}-\lim$ we denote the strong limit of operators and ${\rm s-R-\lim}$ means the strong resolvent limit of operators/linear relations \cite{BHSW2010, Ka}. The Banach space of all bounded operators acting between Hilbert spaces $\cH_1$ and $\cH_2$ is denoted by $\bB(\cH_1,\cH_2)$  and $\bB(\cH):=\bB(\cH,\cH)$.
  The cone of all bounded self-adjoint nonnegative operators in a complex Hilbert space $\cH$ is denoted by $\bB^+(\cH)$. If $A_1,$ $A_2$ are two bounded selfadjoint operators in $\cH$ and the difference $A_1-A_2$ belongs to $\bB^+(\cH)$, then we write $A_1\ge A_2$.
If $\cL$ is a subspace (closed linear manifold) in $\cH$, then $P_\cL$ is the orthogonal projection in $\cH$ onto $\cL$, and $\cL^\perp\stackrel{def}{=}\cH\ominus\cL$.
$\dN$ is the set of natural numbers, $\dN_0:=\dN\cup\{0\}$, $\dC$ is the field of complex numbers, $\dR_+:=[0,\infty)$.
\vskip 2 cm
Let $\cH$ be a complex Hilbert space and let $S\in \bB^+(\cH)$. It was discovered by M.G.~Kre\u{\i}n \cite{Kr} that for an arbitrary subspace $\cL$ of $\cH$ the set
\begin{equation}
\label{maxprob}
\Xi(S,\cL):=\left\{\,\wt S\in\bB^+(\cH):\,
    \wt S\le S, \, {\ran}\wt S\subseteq{\cL}\,\right\}
\end{equation}
has a maximal element $S_{\cL}$.
The operation $\left<S,\cL\right>\mapsto S_\cL$ was applied in \cite{Kr} to the problem of a description contractive selfadjoint extensions of a non-densely defined Hermitian contraction.
The following representations of $S_\cL$ and associated quadratic form were established in \cite{Kr}:
\begin{equation}\label{Krform1}
 S_{\cL}=S^\half P_{\sM}S^\half,\; \left(S_{\cL}f, f\right)=\inf\limits_{\f\in {\cL^\perp}}\left\{\left(S(f + \varphi),f +
 \varphi\right)\right\},
\quad  f\in\cH,
\end{equation}
 where $P_{\sM}$ is the orthogonal projection in $\cH$ onto the subspace $\sM=\left\{f\in\cH:S^\half f\in\cL\right\}.$
Besides \cite{Kr}
\begin{equation}\label{shintr}
 {\ran}S_{\cL}^\half={\ran}S^\half\cap{\cL},\; S_{\cL}=0 \iff  \ran S^\half\cap \cL=\{0\}.
\end{equation}
A bounded selfadjoint operator
$S$ admits the block operator matrix representation w.r.t. the orthogonal decomposition $\cH=\cL^\perp\oplus\cL$:
\[
S=\begin{bmatrix}S_{11}&S_{12}\cr S^*_{12}&S_{22}
\end{bmatrix}:\begin{array}{l}\cL^\perp\\\oplus\\\cL \end{array}\to
\begin{array}{l}\cL^\perp\\\oplus\\\cL \end{array},
\]
where $S^*_{11}=S_{11},$ $S^*_{22}=S_{22}.$
It is well known (see e.g. \cite{Shmul}) that
\begin{equation}\label{POZ}
S\ge 0\Longleftrightarrow \left\{\begin{array}{l}S_{11}\ge 0,\; \ran S_{12}\subset\ran S^\half_{11},\\
S_{22}-\left(S^{[-\half]}_{11}S_{12}\right)^*\left(S^{[-\half]}_{11}S_{12}\right)\ge 0
\end{array}\right.,
\end{equation}
and the operator $S_\cL$ is given by the block matrix
\begin{equation}
\label{shormat1}
S_\cL=\begin{bmatrix}0&0\cr
0&S_{22}-
\left(S^{[-\half]}_{11}S_{12}\right)^*\left(S^{[-\half]}_{11}S_{12}\right)\end{bmatrix},
\end{equation}
 where $S^{[-\half]}_{11}$ is the Moore-Penrose pseudo-inverse.
 If $S^{-1}_{11}\in\bB(\cL^\perp)$,  then the operator $S_{22}-
\left(S^{[-\half]}_{11}S_{12}\right)^*\left(S^{[-\half]}_{11}S_{12}\right)$ takes the form $S_{22}-S^*_{12}S^{-1}_{11}S_{12}$ and is said to be the \textit{Schur complement} of $S$ w.r.t. $\cL$ \cite{Zhang}.
From \eqref{shormat1} it follows that
\[
S_\cL=0\iff \ran S_{12}\subset\ran
S^\half_{11}\quad\mbox{and}\quad S_{22}=S^*_{12}S^{-1}_{11}S_{12}.
\]
The function $\bB^+(\cH)\ni S\mapsto S_\cL$ has the following property \cite{Shmul}: in the strong convergence sense
\begin{equation}\label{ljfd1}
\bB^+(\cH)\supset\{S_n\}\searrow S_\infty\Longrightarrow (S_n)_\cL\searrow (S_\infty)_\cL
\end{equation} 

For the case of a nonnegative $S$ acting on a finite-dimensional Hilbert space $\cH$, the operator $S_\cL$ in \eqref{shormat1}
was called by W.N.~Anderson in \cite{And} the \textit{shorted operator} due to the its connection with electrical network theory. It was shown \cite[Theorem 1]{And} that if the shorted operator is defined by such a way, then it is the maximal element of the set $\Xi(S,\cL)$ defined in \eqref{maxprob}. Basic properties of $S_{\cL}$ were studied in \cite{And, AD, AT,FW, Kr, KrO, Morley, NA, P, P_2014, Shmul}.

 The
\textit{parallel sum} $B:G$ for $B,G\in \bB^+(\cH)$  for the case of finite-dimensional $\cH$ was defined in \cite{AD} as follows
\[
B:G=B(B+G)^{[-1]}G,
\]
where $(B+G)^{[-1]}$ is the Moore-Penrose pseudo-inverse. In the general case of $\dim\cH\le\infty$ the operator $B:G$ is of the form
\begin{equation}\label{bouinv}
B:G=(B^{-1}+G^{-1})^{-1},
\end{equation}
if both $B$ and $G$ have bounded inverses and if this is not the case, then
\begin{multline*}
B:G={\rm s}-\lim\limits_{\varepsilon\downarrow 0}(B+\varepsilon I):(G+\varepsilon I)={\rm s}-\lim\limits_{\varepsilon\downarrow 0}\left((B+\varepsilon I)^{-1}+(G+\varepsilon I)^{-1}\right)^{-1}\\
={\rm s}-\lim\limits_{\varepsilon\downarrow 0}\,
B\left(B+G+\varepsilon I\right)^{-1}G.
\end{multline*}

In \cite{FW} the following important property of the parallel sum is established:
\begin{equation}\label{khfy}
\ran(B:G)^{\half}=\ran B^\half\cap\ran G^\half.
\end{equation}

The quadratic form $((B:G)h,h)$ admits the representation
\[
\left((B:G)\f,\f\right)
=\inf_{f,g \in \cH}\left\{\,\left(Bf,f\right)+\left(Gg,g\right):\,
   \f=f+g \,\right\} \ ,
\]
see \cite{AT, Ando1976, FW,PSh}. It follows that $B:G\le B$ and $B:G\le G$.
As is known \cite{PSh}, $B:G$ can be calculated as follows
\[
B:G=B-\left((B+G)^{[-1/2]}B\right)^*\left((B+G)^{[-1/2]}B\right),
\]
i.e., the parallel addition can be expressed \cite{AT} by means of the shorting:
$$
B:G=\left(\begin{bmatrix}B+G&B\cr B&B \end{bmatrix}\right)_{\{0\}\oplus\cH}\uphar {\{0\}\oplus\cH}.$$
The latter representation leads to the following statement, see \cite[Theorem 2.5]{PSh}:
\begin{multline}\label{ljbfd2}
\bB^+(\cH)\supset\{B_n\}, \{G_n\},\;B_n\searrow B_\infty,\;G_n\searrow G_\infty\quad \mbox{in the strong convergence sense}\\
\Longrightarrow {\rm s}-\lim\limits_{n\to\infty}(B_n:G_n)=B_\infty:G_\infty.
\end{multline}
Another connections between shorting and parallel addition were established in \cite{AT} and are given by the following equalities
\begin{equation}\label{interr}
B_\cL={\rm s}-\lim\limits_{t\uparrow +\infty}(B:tP_\cL),
\end{equation}
\begin{equation}\label{totjlyj}
(B:G)_\cL=B_\cL:G=B:G_\cL=B_\cL: G_\cL.
\end{equation}

Various applications of shorting and parallel addition in complex analysis, operator theory, probability and statistics, numerical analysis are scattered in the literature. For matrices some of applications can be found in the books \cite{MBM,Zhang}.
Generalizations of the shorting and parallel additions to the cases of bounded operators, acting  between two Hilbert spaces, on linear spaces, and on Kre\u{\i}n spaces, are given in  \cite{ACS, CMM, FGK, MP}.
The operation of parallel addition has been extended to quadratic forms in \cite{HSdeS2009, HSdeS2010, Tarcsay2013, Tarcsay2015}.

In this article the shorting and parallel addition are defined and studied for nonnegative selfadjoint linear relations (l.r. for short) and, in particular, for unbounded nonnegative selfadjoint operators. We define the shorting $A_\cL$ similarly to Kre\u{\i}n's definition \eqref{maxprob} for a bounded nonnegative selfadjoint operator. The parallel addition of two nonnegative selfadjoint l.r. we define as in \eqref{bouinv}, replacing the sum by the \textit{form sum}.
Recall that the form sum $A\dot+B$ of nonnegative selfadjoint l.r. $A$ and $B$ is the nonnegative selfadjoint l.r. associated with the closed quadratic form
${\mathfrak g}[\f]=||A^\half \f||^2+||B^\half \f||^2$, $\f\in \dom A^{\half}\cap\dom B^{\half}$ \cite{FM,HSSW06,Ka}. We show, see Theorem \ref{ythdj}, that the inclusion
$\ran A^\half+\ran B^\half\subseteq\ran(A\dot+B)^{\half}$ holds.
In the sequel the nonnegative selfadjoint l.r. $(A\dot+B)/2$ will be called the \textit{arithmetic mean} of $A$ and $B$.
The \textit{harmonic mean} $h(A,B)$ of $A$ and $B$ (in  \cite{K-A} the notation $A!B$ was used) is defined as follows
$$h(A,B)=2(A:B)=\left(\half(A^{-1}\dot +B^{-1})\right)^{-1}.$$
For $A,B\in\bB^+(\cH)$ the inequality $h(A,B)\le (A+B)/2$ is valid \cite{Ando1978}.

For nonnegative selfadjoint l.r. we show that equalities \eqref{shintr}, \eqref{khfy}, \eqref{totjlyj} remain valid and some new properties of shorting and parallel additions we establish in Theorems \ref{shortrel}, \ref{shortre2}, \ref{Frmsm}.

There is a one-to-one correspondence between nonnegative selfadjoint l.r. $L$ and selfadjoint contractions $T$ given by the Cayley transform:
\begin{equation}\label{rtkbtr}
\begin{array}{l}
L\mapsto {{\mathfrak C}}(L)=T:=-I+2(I+L)^{-1}\in [-I, I],\\[3mm]
[-I, I]\ni T\mapsto {{\mathfrak C}}(T)=L:= -I+2(I+T)^{-1}=\left\{\{(I+T)f, (I-T)f\}:f\in\cH\right\},\\[3mm]
\qquad\qquad {{\mathfrak C}}({{\mathfrak C}}(L))=L,\; {{\mathfrak C}}(L^{-1})=-T=-{{\mathfrak C}}(L).
\end{array}
\end{equation}
In Theorem \ref{shortrel} and Theorem \ref{vspom2} we establish the following equalities, connecting the Cayley transforms:
\[
\begin{array}{l}
{{\mathfrak C}}(A_\cL)=I-(I-{{\mathfrak C}}(A))_\cL,\\[3mm]
{\mathfrak C}\left(\half(A\dot +B)\right)= h(I+{\mathfrak C}(A),I+{\mathfrak C}(B))-I,\\[3mm]
{\mathfrak C}\left(h(A,B)\right)= I-h(I-{\mathfrak C}(A),I-{\mathfrak C}(B)),
\end{array}
\]
where $A$ and $B$ are nonnegative selfadjoint l.r. and $\cL$ is a subspace.
Using the above equalities and replacing the strong convergence by the \textit{strong resolvent convergence}, in Proposition \ref{connon} and Corollary \ref{cylbv} we extend properties \eqref{ljfd1}, \eqref{ljbfd2}, \eqref{interr} to the set of all nonnegative selfadjoint l.r..

The mean $c_0(a,b):=\cfrac{a+b+2ab}{2+a+b}$ of positive numbers $a$ and $b$ satisfies the inequalities
\[
\cfrac{2ab}{a+b}\le \cfrac{a+b+2ab}{2+a+b}\le \cfrac{a+b}{2}
\]
with $"="$ if and only if $a= b$.
Using the Cayley transforms, the mean $c_0$ can be extended to arbitrary nonnegative selfadjoint l.r. $A$ and $B$:
\[
c_0(A,B):=\mathfrak C\left(\cfrac{\mathfrak C(A)+\mathfrak C(B)}{2}\right).
\]
We prove in Theorem \ref{ythdj} analogues of the above inequalities:
\[
 h(A, B)\le c_0(A,B)\le\cfrac{A\dot+B}{2}
\]
with $"="$ if and only if $A=B$.

The \textit{geometric mean} $A\#B$ for $A,B\in\bB^+(\cH)$ can be defined as follows \cite{Ando1978, K-A, PuWo}:
\[
A\#B:=\left\{\begin{array}{l}\max\left\{X\in\bB^+(\cH):|(X\f,\psi)|\le ||A^\half\f||\,||B^\half \psi||\;\;\f,\psi\in \cH\right\},\\
A^\half\left(A^{-\half}BA^{-\half}\right)^{\half}A^\half,\; \mbox{if} \;A^{-1}\;\mbox{and}\; B^{-1}\;\mbox{are bounded}
\end{array}\right..
\]
For commuting $A$ and $B$ one has $A\#B=(AB)^\half.$
The following definition \cite{AMT} of geometric mean as the arithmetic--harmonic mean is a straightforward generalization of this notion for positive numbers.
Define
\[
A_0:=A,\; B_0:=B,\; A_{n}:=\half(A_{n-1}+B_{n-1}),\; B_{n}:=h(A_{n-1},B_{n-1}),\; n\in\dN.
\]
Then $\{A_n\},$ $\{B_n\}$ are non-increasing and non-decreasing sequences, respectively, and they have the common strong limit $g(A,B)$, which coincides with $A\#B$ and thus
\[
h(A,B)\le A\#B\le \half(A+B).
\]
 In Section \ref{fhbaufh} we define similar sequences for nonnegative selfadjoint l.r. $A$ and $B$ and show that, in general, there is no a common strong resolvent  limit, i.e., there are two arithmetic-harmonic means $ah(A,B)$.

According to Kre\u{\i}n's results \cite{Kr} and their generalizations in \cite{AN, Ar4, CS,HMS,HSSW06, HSSW07} (see Section \ref{EXTEN}), the Cayley transform of the set of all nonnegative selfadjoint extensions of a nonnegative symmetric l.r. forms an operator interval. In Theorem \ref{ext}, in particular, we find parameters corresponding to the arithmetic and harmonic means of given two nonnegative selfadjoint extensions, in Proposition \ref{extrah} we show that the pair of arithmetic--harmonic means $ah(\wt S_1,\wt S_2)$ for two extremal nonnegative selfadjoint extensions $\wt S_1$, $\wt S_2$ of a nonnegative symmetric operator or linear relation consists of $\half(\wt S_1\dot+\wt S_2)$ and $h(\wt S_1,\wt S_2)$. In Proposition \ref{jcnfy} we show for differential operators in $\cL_2(\dR_+)$
\[
\left\{\begin{array}{l}L_cf=-\cfrac{d^2f}{d x^2}\\
\dom L_c=\left\{f\in W_2^2(\dR_+),\; f'(0)=c f(0)\right\},\; c\in[0,+ \infty],\;W_2^2(\dR_+)\;\mbox{is the Sobolev space}
\end{array}\right.
\]
 the validity of the equalities
\[\begin{array}{l}
\half(L_c\dot+L_d)=L_{\half(c+d)},\; h(L_c, L_d)=L_{\frac{2 c d}{c+d}}, \;ah(L_c,L_d)=L_{\sqrt{cd}},\\[2mm]
\half(L_0\dot+L_\infty)=L_{\infty},\;
 h(L_0, L_\infty)=L_{0}, \;ah(L_\infty,L_0)=\left<L_{\infty}, L_0\right>.
  \end{array}
\]

 In the case of a nonnegative symmetric l.r. having one-dimensional resolvent difference of its Friedrichs and Kre\u{\i}n extensions, we calculate in Theorem \ref{yfltdct}
  parameters corresponding to the resolvents of the arithmetic--harmonic means of any two nonnegative selfadjoint extensions.
 \section{Nonnegative selfadjoint linear relations}
 Let $\cH$ be a Hilbert space and let
\[
J_\cH:\{f,f'\}\mapsto \{if',-if\},\;f,f'\in\cH.
\]
Then $J_\cH$ is selfadjoint and unitary operator in the Hilbert space $\cH^2=\cH\oplus \cH.$
Let $ A$ be a l.r. in $\cH$ \cite{Ar}, i.e., $A$ is a closed linear manifold (a subspace) in the Hilbert space $\cH^2$. Recall that the l.r.
$ A^*:=\cH^2\ominus J_\cH A$
is called the adjoint to $ A$.
If  $ A=\{\{{\f},{\f}'\}\}$, 
then, clearly,
\[
 A^*=\left\{\{{g},{ g}'\}:({ \f}',{ g})=({ \f},{ g}')\;\forall\; \{{ \f},{ \f}'\}\in A\right\}.
\]
 A l.r. $ A$ is called selfadjoint if $ A^*= A$ $\iff$ $ A=\cH^2\ominus J_\cH A$.
 The domain, range, kernel and multi-valued part of a l.r. $A$ are defined as follows:
 \[
 \begin{array}{l}
 \dom A=\left\{f\in\cH:\{f,f'\}\in A\right\}, \;\ran A=\left\{f'\in\cH:\{f,f'\}\in A\right\},\\[3mm]
 \ker A=\left\{f\in\cH:\{f,0\}\in A\right\},\;\mul A=\left\{f'\in\cH:\{0,f'\}\in A\right\}.
\end{array}
 \]
The l.r. $A^{-1}:=\left\{\{f',f\}:\{f,f'\}\in A\right\}$ is called the inverse to $A$.
Clearly
\[
\dom A^{-1}=\ran A, \;\ran A^{-1}=\dom A,\; \ker A^{-1}=\mul A,\;\mul A^{-1}=\ker A.
\]
Recall that a symmetric (selfadjoint) l.r. $A$
admits the orthogonal decomposition 
$A={\rm Graph}(A_{o})\bigoplus \{\{0\}\oplus\mul A\},$
 where $A_{o}$ is the symmetric (selfadjoint) operator part
acting on $\cdom A$ ($\cdom A =\cH\ominus \mul A$ when $A$ is selfadjoint).

The resolvent $(A-\lambda I)^{-1}$ of a selfadjoint l.r. $A$ is defined for all regular points $\rho(A_o)$ of the operator part $A_o$ and takes the form
\[
(A-\lambda I)^{-1}f=(A_o-\lambda I)^{-1}P_Af,\; f\in\cH,
\]
where $P_A$ is the orthogonal projection onto ${\cdom A}$. We will identify $\rho(A)$ with $\rho(A_o)$.

If $\cL$ is a subspace of $\cH$, then $P_\cL(A-\lambda I)^{-1}\uphar\cL$
is called the compressed resolvent of $A$. The holomorphic family $\cM_{A,\cL}$ of l.r. 
\begin{equation}\label{nevfam}
\cM_{A,\cL}(\lambda):=-\left(P_\cL(A-\lambda I)^{-1}\uphar\cL\right)^{-1}-\lambda I_\cL,\; \lambda\in\rho(A_o)
\end{equation}
is a Nevanlinna function or a Nevanlinna family \cite{DM1991,DHMS06}.
Note that
\begin{equation}\label{jdhfn}
\cM_{A^{-1},\cL}(\lambda)=\left(\cM_{A,\cL}(\lambda^{-1})\right)^{-1},\;\lambda\in\rho(A)\setminus\{0\}.
\end{equation}
If a bounded selfadjoint operator $A$ in the Hilbert space $\cH=\cL^\perp\oplus\cL$ is given by the block operator matrix
\[
A=\begin{bmatrix}A_{11}&A_{12}\cr A^*_{12}&A_{22}  \end{bmatrix},\;A_{11}=A^*_{11},\;A_{22}=A^*_{22},
\]
then $\bB(\cL)$-valued function
\begin{equation}\label{nevbound}
\cM_{A, \cL}(\lambda):=-A_{22}+A_{12}^*\left(A_{11}-\lambda I_{\cL^\perp}\right)^{-1}A_{12},\;\lambda\in\rho(A_{11})
\end{equation}
is a Nevanlinna function and
\[
P_{\cL}\left(A-\lambda I_\cL\right)^{-1}\uphar\cL=-\left(\cM_{A,\cL}(\lambda)+\lambda I_\cL\right)^{-1},\; \lambda\in\rho(A)\cap\rho(A_{11}).
\]

In the sequel the notation $A\uphar\cL$ is used for the l.r.
\[
A\uphar\cL:=\left\{\{f,f'\}: f\in\dom A\cap \cL,\;\{f,f'\}\in A\right\}.
\]
It is established in \cite{Stenger} that if $A$ is a selfadjoint unbounded linear operator and $\cL$ is a \textit{subspace with finite codimension}, then
the operator $P_\cL A\uphar\cL$ is selfadjoint in the space $\cL$. In \cite{Nud} has been proved that the compression $P_\cL A\uphar\cL$ of a maximal dissipative operator $A$ is maximal dissipative in $\cL$. Further in \cite{ADW} similar assertions were obtained for selfadjoint and maximal dissipative l.r..

Let $A=\left\{\{f,f'\}\right\}$ be a symmetric l.r.
in the Hilbert space $\cH$. Then for any $\{f,f'\}\in A$ one has the equality \cite{RoBe}
\[
(f',f)=(A_o f,f).
\]
A l.r. $A$ is called  nonnegative (we will write $A\ge 0$) if $(f',f)\ge 0$ for all $\{f,f'\}\in A$.

If $A$ is a nonnegative  selfadjoint l.r. then the square root $A^\half$ is defined as follows
$$A^{\half}={\rm Graph}(A_{o}^{\half})\bigoplus \{\{0\}\oplus \mul A\}\},\; \ran A^{\half}= \ran A_{o}^{\half}\oplus\mul A.$$
Hence
$$A^{-\half}=\{\{A_{o}^{\half}f,f\},\; f\in\dom A^\half_o\}\bigoplus \{\{\mul A\oplus\{0\}\}.$$
Let ${\mathfrak a}={\mathfrak a}[\cdot,\cdot]$ be a nonnegative sesquilinear form in the Hilbert
space $\sH$ with domain $\dom {\mathfrak a}$ and let ${\mathfrak a}[\f]:={\mathfrak a}[\f,\f]$, $\f \in \dom {\mathfrak a}$. The form ${\mathfrak a}$ is
\textit{closed} \cite{Ka} if
\[
\left\{ \begin{array}{l}\lim\limits_{n\to \infty}\f_n=\f, \quad \lim\limits_{n,m\to\infty}{\mathfrak a}[\f_n-\f_m]=0,\\
\quad \{\f_n\}\subset \dom {\mathfrak a}\end{array}\right.\Longrightarrow \f \in \dom {\mathfrak a},\;\lim\limits_{n\to\infty}{\mathfrak a}[\f_n-\f]=0.
\]
The form ${\mathfrak a}
$ is \textit{closable} \cite{Ka} if
\[
\left\{ \begin{array}{l}\lim\limits_{n\to \infty}\f_n=0, \quad \lim\limits_{n,m\to\infty}{\mathfrak a}[\f_n-\f_m]=0,\\
\quad \{\f_n\}\subset \dom {\mathfrak a}\end{array}\right.\Longrightarrow
\quad \lim\limits_{n\to\infty} {\mathfrak a}[\f_n]=0.
\]
The form ${\mathfrak a} $ is closable if and only if it has a closed
extension, and in this case the closure of the form is the smallest
closed extension of ${\mathfrak a}$. The inequality ${\mathfrak a}_1 \geq {\mathfrak a}_2$ for
semi-bounded forms ${\mathfrak a}_1$ and ${\mathfrak a}_2$ is defined by
\begin{equation}\label{ineq0}
 \dom {\mathfrak a}_1 \subseteq \dom {\mathfrak a}_2, \quad {\mathfrak a}_1[\f] \geq
{\mathfrak a}_2[\f], \quad \f \in \dom {\mathfrak a}_1.
\end{equation}
In particular, ${\mathfrak a}_1 \subset {\mathfrak a}_2$ implies ${\mathfrak a}_1 \geq {\mathfrak a}_2$.
If the forms ${\mathfrak a}_1$ and ${\mathfrak a}_2$ are closable, the inequality
${\mathfrak a}_1 \geq {\mathfrak a}_2$ is preserved by their closures.

If $A=\left\{\{f, f'\}\right\}$ is a nonnegative symmetric l.r., then the form
$${\mathfrak a}[f]= (f',f),\; \{f, f'\}\in A,\; f\in\dom A=\dom A_o$$
 is closable \cite{Ka,RoBe}.
There is a one-to-one correspondence between all closed nonnegative
forms ${\mathfrak a}$ and all nonnegative selfadjoint l.r. $A$ in
$\sH$ (the first representation theorem \cite{Ka}), see \cite{Ka,RoBe}, via $\dom A \subset \dom {\mathfrak a}$ and
\begin{equation}\label{einz}
 {\mathfrak a}[f,\psi]=(A_{o} f,\psi), \quad f \in \dom A, \quad \psi \in \dom {\mathfrak a}.
\end{equation}
In what follows the closed form associated with $A$ is
denoted by $A[\cdot,\cdot]$ and its domain by $\cD[A]$.
By the second representation theorem
\begin{equation}\label{zwei2}
 A[\f,\psi]=(A_{o}^\half \f, A_{o}^\half \psi),  \quad \f,\psi \in \cD[A]=\dom A_{o}^\half.
\end{equation}
The formulas \eqref{einz}, \eqref{zwei2} are analogs of Kato's
representation theorems for, in general, nondensely defined closed
semi-bounded forms in \cite[Section~VI]{Ka}; see e.g.
\cite{RoBe,AHZS,HSSW07}.

If $A$ is a nonnegative selfadjoint l.r. and if $T={{\mathfrak C}}(A)$ is its Cayley transform, then \cite{ArlBelTsek2011}
\begin{equation}\label{AJHVF}
\left\{ \begin{array}{l}\cD[ A]=\ran (I+T)^\half,\\[3mm]
A[u,v]=-(u,v)+2\left(I+ T)^{[-\half]}u,(I+T)^{[-\half]}v\right),\; u,v\in\cD[A],
\end{array}\right.
\end{equation}
where $(I+T)^{[-\half]}$ is the Moore-Penrose pseudo-inverse. Replacing $A$ by $A^{-1}$ in \eqref{AJHVF}, we conclude that $\ran A^{\half}=\ran (I-T)^{\half}.$

Given a sequence $\{A_n\}$ of nonnegative selfadjoint l.r., we say that $\{A_n\}$ converges in the strong resolvent sense \cite{Ka} to a nonnegative selfadjoint l.r. $A$\\
(${\rm s-R}-\lim\limits_{n\to\infty}A_n=A$) if the sequence of the resolvent $\{(A_n+I)^{-1}\}$ converges in the strong sense to the resolvent $(A+I)^{-1}$ (${\rm s}-\lim\limits_{n\to\infty}(A_n+I)^{-1}=(A+I)^{-1}$).

Let $A_1$ and $A_2$ be nonnegative selfadjoint l.r. in $\cH$,
then $A_1$ and $A_2$ are said to satisfy the inequality $A_1 \ge
A_2$ if
\[
 \dom A_{1,o}^{\half}\subset \dom A_{2,o}^{\half}\text{ and }
 \|A_{1,o}^{\half} \f\| \ge \|A_{2,o}^{\half} \f\|, \quad
 \f \in \dom A_{1,o}^{\half}.
\]
This means that the closed nonnegative forms $A_1[\cdot,\cdot]$ and
$A_2[\cdot,\cdot]$ associated with $A_1$ and $A_2$ satisfy the
inequality $A_1[\f] \geq A_2[\f]$ for all $\f\in D[A_1]$; see \eqref{ineq0}, \eqref{zwei2}.
Finally we note that the inequalities $A_1\ge A_2\ge 0$ are equivalent to the inequality
$A^{-1}_1\le A^{-1}_2$, see \cite{HSSW06}.
Besides,
\begin{equation}\label{yjdjtyt}
A_1\ge A_2\ge 0\Longleftrightarrow -I\le{{\mathfrak C}}(A_1)\le {{\mathfrak C}}(A_2)\le I.
\end{equation}
Let $A$ be a nonnegative selfadjoint l.r. in $\cH$. Then the resolvent $(A+x I)^{-1}$ is a bounded nonnegative selfadjoint operator for an arbitrary positive number $x$.
If $0\le x_1\le x_2$, then, clearly, $ A^{-1}\ge (A+x_1 I)^{-1}\ge (A+x_2 I)^{-1}$ and
\begin{equation}\label{ranhalf}\begin{array}{l}
\lim\limits_{x\downarrow 0}\left((A+x I)^{-1}f,f\right)=\lim\limits_{x\downarrow 0}\left\|(A_{o}+x I)^{-\half}P_Af\right\|^2\\
=\left\{\begin{array}{l}0,\qquad\qquad\qquad f\in\mul A\\
||A_{o}^{[-\half]}P_Af||^2,\; P_Af\in\ran A^{\half}_{o}\\
+\infty, \qquad\qquad f\in \cH\setminus\ran A^{\half}
\end{array}\right.=\left\{\begin{array}{l} A^{-1}[f],\qquad f\in\ran A^{\half}\\
+\infty, \qquad\qquad f\in \cH\setminus\ran A^{\half}
\end{array}\right..
\end{array}
\end{equation}

Let $A$ and $B$ be two nonnegative selfadjoint l.r.. If $\dom A\cap\dom B\ne 0$, then it is naturally defined the sum
\[
A+B:=\left\{\{f,f'+g'\}: \; \{f,f'\}\in A,\;\{f,g'\}\in B\right\}
\]
which is nonnegative symmetric l.r. and it, in general, is not selfadjoint.

On the other side, one can define a quadratic form
\[
{\mathfrak g}[\f]=A[\f]+B[\f],\; \f\in \cD[A]\cap\cD[B],
\]
which is nonnegative and closed \cite{Ka,FM,HSSW06}. Hence, by the first representation theorem, \cite{Ka, RoBe} there is a nonnegative selfadjoint l.r. associated with ${\mathfrak g}$. This l.r. is called the form sum of $A$ and $B$ \cite{Faris, FM, HSSW06, Ka} and is  denoted by $A\dot+ B$. Observe that
 $$A\dot+B=\{0\}\oplus\cH\Longleftrightarrow\cD[A]\cap\cD[B]=\{0\}.$$

\section{Shorted operators for non-negative selfadjoint linear relations}

\begin{theorem}\label{shortrel}
Let $A$ be a nonnegative selfadjoint l.r. in the Hilbert space $\cH$ and let $\cL$ be a subspace of $\cH$.
Then the set 
\begin{equation}\label{setxi}
\Xi(A,\cL)=\left\{\wt A \;\mbox{is a nonnegative selfadjoint l.r.},\;\wt A\le A,\;\ran \wt A\subseteq \cL\right\}
\end{equation}
has a unique maximal element $A_\cL$ and $\ker A_\cL\supseteq\cL^\perp$. Moreover,
\begin{enumerate}
\item
the l.r. $A_\cL\uphar\cL$ is selfadjoint in $\cL$ and $(A_\cL\uphar\cL)^{-1}$ is associated with the closed sesquilinear form ${\mathfrak c_\cL}$ defined as follows
\begin{equation}\label{sesqk}
{\mathfrak c_\cL}[f,g]:= A^{-1}[f,g],\;f,g\in \dom {\mathfrak c_\cL}=\ran A^\half\cap \cL;
\end{equation}
\item the equality $\ran(A_\cL)^{\half}=\ran A^{\half}\cap\cL$ holds and $A_\cL=0$ ($\dom A_\cL=\cH$) if and only if $\ran A^{\half}\cap\cL=\{0\}$;
\item if ${\mathfrak C}(A)$ is the Cayley transform \eqref{rtkbtr} of $A,$ then
\begin{equation}\label{gthda1}
 A_\cL= {\mathfrak C}\left(I-(I-{\mathfrak C}(A))_\cL\right);
\end{equation}
\item $A_\cL$ is an operator ($\mul A_\cL=\{0\}$) if and only if $\mul A\cap\cL=\{0\}$;
\item $A_\cL\in\bB^+(\cH)\setminus\{0\}$ if and only if the linear manifold
\begin{equation}\label{ghjcn}
\sM:=\left\{f\in\cD[A]:A^{\half}f\in\cL\right\}
\end{equation}
is a subspace and $\mul A\cap\cL=\{0\}$.
\end{enumerate}
\end{theorem}
\begin{proof}
The operator $T:={\mathfrak C}(A)$ admits the following block operator matrix representation
\begin{equation}\label{cayltr}
T=\begin{bmatrix} D&C\cr
C^*&F\end{bmatrix}:\begin{array}{l}\cL^{\perp}\\\oplus\\
\cL\end{array}\to
\begin{array}{l}\cL^{\perp}\\\oplus\\
\cL\end{array}.
\end{equation}

If $\wt A\in\Xi(A,\cL)$ and if $\wt T={{\mathfrak C}}(\wt A)=-I+2(I+\wt A)^{-1}$ is the Cayley transform of $\wt A$, then because $\wt A+I\le A+I$ and
$2(\wt A+I)^{-1}=I+\wt T$, $2(A+I)^{-1}=I+T$, we conclude that $I+\wt T\ge I+T.$
On the other hand, the representation
\[
\wt A=\left\{\{(I+\wt T)g,(I-\wt T)g\},\; g\in\cH\right\}
\]
and the inclusion $\ran \wt A\subseteq\cL$ yield that $\ran (I-\wt T)\subseteq\cL.$
The latter is equivalent to the equality $\wt T\uphar\cL^\perp=I_{\cL^\perp}.$
It follows that $\wt T$ is of the form
$ \wt T=\begin{bmatrix} I&0\cr
0&\wt F\end{bmatrix}:\begin{array}{l}\cL^{\perp}\\\oplus\\
\cL\end{array}\to
\begin{array}{l}\cL^{\perp}\\\oplus\\
\cL\end{array}.$ 
Since
\[
(I+\wt T)-(I+T)=\begin{bmatrix} I-D&-C\cr
-C^*&\wt F-F\end{bmatrix}:\begin{array}{l}\cL^{\perp}\\\oplus\\
\cL\end{array}\to
\begin{array}{l}\cL^{\perp}\\\oplus\\
\cL\end{array}
\]
and $(I+\wt T)-(I+T)\ge 0$, from \eqref{POZ} we get the following inequality
\[
\wt F-F-\left( (I-D)^{[-\half]}C\right)^*(I-D)^{[-\half]}C\ge 0.
\]
Recall (cf. \cite{AG,DaKaWe,ShYa}), that the block operator-matrix \eqref{cayltr}
is a selfadjoint contraction if and only if the following properties of the entries are valid
\begin{equation}\label{param}
\left\{\begin{array}{l}
D\in [-I_{\cL^\perp}, I_{\cL^\perp}],\; C^*=ND_D, F=-NDN^*+D_{N^*}XD_{N^*},\\[3mm]
N\in \bB(\sD_D,\cL)\quad \mbox{is a contraction},\\
 X\in\bB(\sD_{N^*})\quad\mbox{is a selfadjoint contraction}
 \end{array}\right..
\end{equation}

From the structure \eqref{param} of entries of $T$ we obtain that
\[
F':=F+\left( (I-D)^{[-\half]}C\right)^*(I-D)^{[-\half]}C=NN^*+D_{N^*}XD_{N^*}.
\]
Thus, $-I_\cL\le F'\le\wt F\le I_\cL.$
Since the set $\Xi(A,\cL)$ coincides with the Cayley transforms of the set
\[
\Upsilon(T,\cL):={{\mathfrak C}}\left(\Xi(A,\cL)\right)=\left\{\wt T\in[-I, I]:\wt T\uphar\cL^\perp=I_{\cL^\perp}, \wt T\ge T\right\}
\]
and
\[
\min\limits_{\wt T\in\Upsilon(B\cL)}\wt T= T'=\begin{bmatrix} I&0\cr
0& F'\end{bmatrix}:\begin{array}{l}\cL^{\perp}\\\oplus\\
\cL\end{array}\to
\begin{array}{l}\cL^{\perp}\\\oplus\\
\cL\end{array},
\]
we obtain that
\[
\max\limits_{\wt A\in \Xi(A,\cL)}\wt A={{\mathfrak C}}(T')=:A_\cL.
\]
Let $(I-T)_\cL$ be the Kre\u{\i}n shorted operator. Then \eqref{cayltr} and \eqref{shormat1} yield
\[
(I- T)_\cL=\begin{bmatrix}0&0\cr 0& I-F-\left( (I-D)^{[-\half]}C\right)^*(I-D)^{[-\half]}C  \end{bmatrix}=
\begin{bmatrix}0&0\cr 0& I-F' \end{bmatrix}.
\]
Therefore
\[
T'=I-(I-T)_\cL,\; A_\cL={{\mathfrak C}}\left(I-(I-T)_\cL\right).
\]

Due to the properties of the shorted operator we have
\[
\cL\cap \ran(I- T)^\half=\ran\left((I-T)_\cL\right)^\half=\ran(I-F')^\half.
\]
It follows that
\[
\begin{array}{l}
A_\cL={{\mathfrak C}}(T')=\left\{\{(I+T')g,(I-T')g\},\; g\in\cH\right\}\\[3mm]
\qquad=\left\{\{h,0\}: h\in\cL^\perp\right\}\bigoplus\left\{\{(I_\cL+F')f,(I_\cL-F')f\}:f\in \cL\right\}=\left(\cL^\perp\oplus\{0\}\right)\bigoplus{{\mathfrak C}}(F').
\end{array}
\]
Hence $A_\cL\uphar\cL={{\mathfrak C}}(F')\subset\cL\oplus\cL.$
Set
\[
\cL':=\{f\in
\cH:(I-T)^\half f\in \cL\}.
\]
Because (see \eqref{Krform1}),
\[
(I-T)_\cL=(I-T)^\half P_{\cL'}(I-T)^{\half}
\]
and
\[
(I-T)_\cL\uphar \cL=I-F'
\]
we get that
\[
(I-T)^{\half}\uphar \cL'=(I-F')^{\half}V',
\]
 where $V'$ is an isometry from $\cL'$ onto $\cran(I-F')$.
It follows that
\[
||(I-T)^{[-\half]}f||=||(I-F')^{[-\half]}f||,\; f\in\ran (I-T)^{\half}\cap \cL=\ran (I-F')^{\half}.
\]
Now \eqref{AJHVF} yields that
\[
 A^{-1}[f,g]=\left({{\mathfrak C}}(F')\right)^{-1}[f,g]=\left(A_\cL\uphar\cL\right)^{-1}[f,g],\; f,g\in\cD[ A^{-1}]\cap \cL=\ran A^{\half}\cap\cL.
\]
Using definition \eqref{sesqk} of the form ${\mathfrak c_\cL}$, we get
$$\mul A_\cL=\{0\}\Longleftrightarrow \ker (A_\cL)^{-1}=\{0\}\Longleftrightarrow \ker{\mathfrak c_\cL}=\{0\}\Longleftrightarrow \mul A\cap\cL=\{0\},$$
$$A_\cL\in\bB^+(\cH)\setminus\{0\}\Longleftrightarrow \left\{\begin{array}{l}\mul A\cap\cL=\{0\},\\
{\mathfrak c_\cL}[f]\ge m||f||^2\;\forall f\in\cL\cap\ran A_o^\half,\; m>0   \end{array}\right.. $$
The latter is equivalent to the conditions: $\mul A\cap\cL=\{0\}$ and the linear manifold defined in \eqref{ghjcn} is closed.

\end{proof}
Observe that since $A_\cL\le A$ and $\cD[A_\cL]=\cL^\perp\oplus D[A_\cL\uphar\cL]$, we have
\[
D[A]\cap\cL\subseteq D[A_\cL\uphar\cL],\; A_\cL[f]\le A[f]\;\;\forall f\in D[A]\cap\cL.
\]
Let $A_1$ and $A_2$ be nonnegative selfadjoint l.r.. Suppose $A_1\le A_2$. Then $\Xi(A_1,\cL)\subset\Xi(A_2,\cL)$ for any subspace $\cL$ in $\sH$.
Hence
\[
(A_1)_\cL\le (A_2)_\cL.
\]
The next proposition is an extension of the property \eqref{ljfd1} to the set of all nonnegative selfadjoint l.r..
\begin{proposition}\label{connon}
Let $\{A_n\}$  be a sequences of nonnegative selfadjoint l.r. and let $\cL$ be a subspace in $\cH$. Then in the strong resolvent convergence sense
$$A_n\searrow A_\infty\Longrightarrow(A_n)_\cL\searrow (A_\infty)_\cL.$$
\end{proposition}

\begin{proof}
Set $T_n:=\mathfrak C(A_n)$, $n\in\dN$,\; $ T_\infty:=\mathfrak C(A_\infty)$. Then
\[
{\rm s-R}-\lim\limits_{n\to\infty}A_n=A_\infty\Longrightarrow {\rm s}-\lim\limits_{n\to \infty}T_n=T_\infty.
\]
Moreover, $\{A_n\}$ is non-increasing $\Longleftrightarrow$ $\{T_n\}$ is non-decreasing $\Longleftrightarrow$ $\{I-T_n\}$ is non-increasing $\Longrightarrow$
$\{(I-T_n)_\cL\}$ is non-increasing for each subspace $\cL$ in $\cH$. Because $\{I-T_n\}\subset\bB^+(\cH)$ we get \cite{Shmul} that in the strong convergence sense
\[
(I-T_n)_\cL\searrow (I-T_\infty)_\cL.
\]
Hence from \eqref{gthda1}
\[
{\rm s-R}-\lim\limits_{n\to\infty} (A_n)_\cL={\rm s-R}-\lim\limits_{n\to\infty} \mathfrak C\left(I-(I-T_n)_\cL\right)=\mathfrak C\left(I-(I-T_\infty)_\cL\right)
=(A_\infty)_\cL.
\]
The proof is complete.
\end{proof}
Replacing $A$ by $A^{-1}$ in Theorem \ref{shortrel} we get that
\[
(A^{-1})_\cL=\max\{\wt B:\wt B\in \Xi(A^{-1},\cL)\}
=\min\left\{\wt A:\wt A=\wt A^*,\;\wt A\ge A,\;\dom \wt A\subseteq\cL\right\}
\]
and $((A^{-1})_\cL\uphar\cL)^{-1}$ is associated with the closed form
\begin{equation}\label{deka}
{\mathfrak d}_\cL[\varphi,\psi]=A[\varphi,\psi],\; \varphi,\psi\in\cD [A]\cap\cL.
\end{equation}
Besides, $\ran((A^{-1})_\cL)^{\half}=\cD [A]\cap \cL$ and $(A^{-1})_\cL=0$ if and only if $\cD [A]\cap \cL=\{0\}.$
\begin{theorem} \label{shortre2}
Let $A$ be a nonnegative selfadjoint l.r. in the Hilbert space $\cH$ and let $\cL$ be a subspace of $\cH$. Then
\[
A_\cL\le ((A^{-1})_\cL)^{-1}.
\]
Let $\cM_{A,\cL}(\lambda)$  be the Nevanlinna family defined in \eqref{nevfam}.
Then
\[
\begin{array}{l}
A_\cL\uphar\cL=-\left({\rm s-R}-\lim\limits_{\lambda\uparrow 0}\cM_{A,\cL}(\lambda)\right),\\[2mm]
((A^{-1})_\cL)^{-1}\uphar\cL=
-\left({\rm s-R}-\lim\limits_{\lambda\downarrow -\infty}\cM_{A,\cL}(\lambda)\right).
\end{array}
\]

\end{theorem}
\begin{proof}
If ${{\mathfrak C}}(A)=T$, then ${{\mathfrak C}}(A^{-1})=-T$. Hence from \eqref{gthda1} we get
\begin{equation}\label{gthda2}
(A^{-1})_\cL={{\mathfrak C}}\left(I-(I+T)_\cL\right).
\end{equation}
Then \eqref{gthda2} implies
\begin{equation}\label{gthda3}
\left((A^{-1})_\cL\right)^{-1}={{\mathfrak C}}\left((I+T)_\cL-I\right).
\end{equation}
The equality $(I-T)+(I+T)=2I$ and the definition of the shorted operator produce the inequality
\[
(I-T)_\cL+(I+T)_\cL\le 2 I.
\]
Hence
\[
I-(I-T)_\cL\ge (I+T)_\cL-I.
\]
From \eqref{yjdjtyt} it follows that
\[
A_\cL={{\mathfrak C}}\left(I-(I-T)_\cL\right)\le {{\mathfrak C}}\left((I+T)_\cL-I\right)=\left((A^{-1})_\cL\right)^{-1}.
\]
 The operator $T$ takes the block operator matrix form \eqref{cayltr}. The entries $C$ and $F$ of $T$ admit the representations \eqref{param}, i.e.,
\[
C^*=ND_D,\; C=D_DN^*,\; F=-NDN^*+D_{N^*}XD_{N^*}.
\]
Define the function
\begin{equation}
\label{strau}
\Phi(z)=F+zC^*(I_{\cL^\perp}-zD)^{-1}C,\; z\in\dC\setminus \{(-\infty,-1]\cup[1,+\infty)\}.
\end{equation}
The function $\Phi$ belongs to the class $\cRS(\cL)$ \cite{ArlHassi_2019}, i.e., $\Phi$ is a Nevanlinna function and the Schur class function in the unit disk $\dD$.
Then there exist strong limit values $\Phi(\pm 1)$ and
\[
\begin{array}{l}
\Phi(-1)=s-\lim\limits_{x\downarrow-1}\Phi(x)=F-\left( (I+D)^{[-\half]}C\right)^*(I+D)^{[-\half]}C=-NN^*+D_{N^*}XD_{N^*},\\[3mm]
\Phi(1)=s-\lim\limits_{x\uparrow 1}\Phi(x)=F+\left( (I-D)^{[-\half]}C\right)^*(I-D)^{[-\half]}C=NN^*+D_{N^*}XD_{N^*}.
\end{array}
\]
Observe that
\[
(I\pm T)_\cL=\begin{bmatrix}0&0\cr 0& I\pm\Phi(\mp 1)\end{bmatrix}:\begin{array}{l}\cL^{\perp}\\\oplus\\
\cL\end{array}\to
\begin{array}{l}\cL^{\perp}\\\oplus\\
\cL\end{array}.
\]
Therefore \eqref{gthda1} and \eqref{gthda3} yield
\[
A_\cL\uphar\cL={{\mathfrak C}}(\Phi(1)),\; ((A^{-1})_\cL)^{-1}\uphar \cL={{\mathfrak C}}(\Phi(-1)).
\]
But then due to $\Phi(1)={{\mathfrak C}}(A_\cL\uphar\cL)$ and $\Phi(-1)={{\mathfrak C}}(((A^{-1})_\cL)^{-1}\uphar \cL)$ we have $-\Phi(1)={{\mathfrak C}}\left((A_\cL\uphar\cL)^{-1}\right),$
\[
I_\cL-\Phi(1)=2(I_\cL+(A_\cL\uphar\cL)^{-1}\uphar\cL)^{-1},\;
I_\cL+\Phi(-1)=2\left(I_\cL+((A^{-1})_\cL)^{-1}\uphar \cL\right)^{-1}
\]
Note that according to the Schur-Frobenius formula for the resolvent of block-operator matrix the relation
\begin{equation}\label{compres1}
P_\cL(I_\cH-z T)^{-1}\uphar\cL=(I_\cL-z\Phi(z))^{-1},\; z\in\dC\setminus \{(-\infty,-1]\cup[1,+\infty)\}
\end{equation}
holds.

Further we  will use the following notations
\[
\begin{array}{l}
\cN_{A,\cL}(\lambda):=P_\cL(A-\lambda I)^{-1}\uphar \cL,\;
\lambda\in\dC\setminus\dR_+,\\[3mm]
\cN_{T, \cL}(\xi):=P_\cL(T-\xi I)^{-1}\uphar \cL,\;
\xi\in\dC\setminus[-1,1].
\end{array}
\]
The resolvents of $A$ and $T$ are connected by the relations
\[
\begin{array}{l}
\left\{\begin{array}{l}
 (A-\lambda
I)^{-1}=-\frac{1}{1+\lambda}\left(I+\frac{2}{1+\lambda}\left( T-\frac{1-\lambda}{1+\lambda}\, I\right)^{-1}\right),\;\lambda\in\rho(A)\setminus\{-1\}\\[3mm]
(A+I)^{-1}=\cfrac{1}{2}\left(T+I\right),\end{array}\right.,\\[3mm]
\Longleftrightarrow
(A-\lambda
I)^{-1}=\cfrac{1}{1-\lambda}(T+I)\left(I-\frac{1+\lambda}{1-\lambda}\,T\right)^{-1},\;\lambda\in\rho( A).
\end{array}
\]
Hence
\begin{multline}\label{cjj1}
\left(\cN^{-1}_{A, \cL}(\lambda)+I\right)^{-1}=\cN_{A, \cL}(\lambda)\left(\cN_{A,\cL}(\lambda)+I\right)^{-1}\\
=-\left(I+\frac{2}{1+\lambda}\cN_{T, \cL}\left(\frac{1-\lambda}{1+\lambda}\right)\right)\,\left(\lambda-\frac{2}{1+\lambda}\cN_{T, \cL}\left(\frac{1-\lambda}{1+\lambda}\right)\right)^{-1}\\
=-\left(\cN^{-1}_{T, \cL}\left(\frac{1-\lambda}{1+\lambda}\right)+\frac{2}{1+\lambda}\right)\left(\lambda
\cN^{-1}_{T, \cL}\left(\frac{1-\lambda}{1+\lambda}\right)-\frac{2}{1+\lambda}\right)^{-1}.
\end{multline}
From the equality ${{\mathfrak C}}(A^{-1})=-T$
and the equality
\[\cN^{-1}_{-T, \cL}\left(\frac{1-\lambda}{1+\lambda}\right)=-\cN^{-1}_{T, \cL}\left(-\frac{1-\lambda}{1+\lambda}\right)
\]
one gets
\begin{multline}\label{cjj2}
\left(\cN^{-1}_{A^{-1}, \cL}(\lambda)+I\right)^{-1}\\
=-\left(-\cN^{-1}_{T, \cL}\left(-\frac{1-\lambda}{1+\lambda}\right)+\frac{2}{1+\lambda}\right)\left(-\lambda
\cN^{-1}_{T, \cL}\left(-\frac{1-\lambda}{1+\lambda}\right)-\frac{2}{1+\lambda}\right)^{-1}.
\end{multline}
Using \eqref{compres1} in the form
\[
\cN_{ T, \cL}(\xi)=(\Phi(\xi^{-1})-\xi I)^{-1},\; \xi\in\dC\setminus [-1,1],
\]
where $\Phi(\cdot)$ is defined by \eqref{strau}, we get the equalities
\[
s-\lim\limits_{\xi\downarrow 1}\cN^{-1}_{T, \cL}(\xi)
=\Phi(1)-I_\cL,\;
s-\lim\limits_{\xi\uparrow -1}
\cN^{-1}_{T, \cL}(\xi)
=I_\cL+\Phi(-1).
\]
Using \eqref{cjj1}, \eqref{cjj2} and
since $$ 2(I+(A_\cL)^{-1}\uphar\cL)^{-1}=I-\Phi(1),\;2(I+((A^{-1})_\cL)^{-1}\uphar\cL)^{-1}=I+\Phi(-1),$$
we have
\[
\begin{array}{l}
A_\cL\uphar\cL={\rm s-R}-\lim\limits_{\lambda\uparrow 0}\left(P_\cL(A-\lambda)^{-1}\uphar \cL \right)^{-1},\\
(A^{-1})_\cL={\rm s-R}-\lim\limits_{\lambda\uparrow 0}\left(P_\cL( A^{-1}-\lambda)^{-1}\uphar \cL \right)^{-1}.
\end{array}
\]
From \eqref{nevfam}
\[
\left(P_\cL(A-\lambda)^{-1}\uphar \cL \right)^{-1}=-\cM_{A,\cL}(\lambda)-\lambda I_\cL,\;\lambda\in\rho(A).
\]
Hence
\[
A_\cL\uphar\cL=-\left({\rm s-R}-\lim\limits_{\lambda\uparrow 0}\cM_{A,\cL}(\lambda)\right).
\]
Similarly
\[
(A^{-1})_\cL\uphar\cL=-\left({\rm s-R}-\lim\limits_{\lambda\uparrow 0}\cM_{A^{-1},\cL}(\lambda)\right).
\]
Consequently
\[
\left(((A^{-1})_\cL)\uphar\cL\right)^{-1}=-\left({\rm s-R}-\lim\limits_{\lambda\uparrow 0}\left(\cM_{A^{-1},\cL}(\lambda)\right)^{-1}\right).
\]
Now \eqref{jdhfn} yields that
\begin{multline*}
\left(((A^{-1})_\cL\uphar\cL\right)^{-1}=-\left({\rm s-R}-\lim\limits_{\mu \downarrow -\infty}\cM_{A,\cL}(\mu)\right)\\
={\rm s-R}-\lim\limits_{\mu\downarrow -\infty}\left(\mu I_\cL+\left(P_\cL(A-\mu I)^{-1}\uphar \cL \right)^{-1}\right).
\end{multline*}
\end{proof}
Let $
A=\begin{bmatrix}A_{11}&A_{12}\cr A^*_{12}&A_{22} \end{bmatrix}:\begin{array}{l}\cL^{\perp}\\\oplus\\
\cL\end{array}\to
\begin{array}{l}\cL^{\perp}\\\oplus\\
\cL\end{array}
$
be a bounded nonnegative selfadjoint operator in $\cH$.
Due to \eqref{POZ} the entry $A_{12}$ admits the representation
\[
A_{12}=A^{\half}_{11}YA^{\half}_{22},\; Y\in\bB(\cran A_{22},\cran A_{11})\quad \mbox{is a contraction}.
\]
From \eqref{shormat1} we get the equality $A_\cL=A^{\half}_{22}(I-Y^*Y)A^{\half}_{22}.$
In general $A^{-1}$ is a l.r.. But $((A^{-1})_\cL)^{-1}$ is associated with the form $A[\varphi,\psi]$, $\varphi,\psi\in\cL$.
Hence
\[
((A^{-1})_\cL\uphar\cL)^{-1}=A_{22}.
\]
So, if $A_{22}$ has bounded inverse, then $(A^{-1})_\cL$ is bounded and vice versa. If this is a case, then $(A^{-1})_\cL=A^{-1}_{22}.$

On the other side the function $\cM_{A,\cL}$ defined in \eqref{nevbound} takes the form
\[
\cM_{A,\cL}(\lambda)=-A_{22}+A^{\half}_{22}Y^*(A_{11}-\lambda I_{\cL^\perp})^{-1}A_{11}YA^{\half}_{22},\;\lambda\in \rho(A_{11}).
\]
Hence
\[
\begin{array}{l}
-\left({\rm s}-\lim\limits_{\lambda\uparrow 0}\cM(\lambda)\right)=A^{\half}_{22}(I-Y^*Y)A^{\half}_{22}=A_\cL\uphar\cL,\\
-\left({\rm s}-\lim\limits_{\lambda\downarrow -\infty}\cM(\lambda)\right)=A_{22}=((A^{-1})_\cL\uphar\cL)^{-1}.
\end{array}
\]
\begin{remark}\label{steng}
Suppose that $\codim\cL<\infty.$
Let $A$ be a nonnegative selfadjoint l.r. in $\cH$. Then due to \cite{Stenger, ADW} the l.r. $P_\cL A^{-1}\uphar\cL$ and $P_\cL A\uphar\cL$ are selfadjoint and, therefore, they are associated with the closed forms ${\mathfrak c_\cL}$ and ${\mathfrak d}_\cL$ defined in \eqref{sesqk} and in \eqref{deka}, respectively.
Hence, due to Theorem \ref{shortrel} one gets
\[
\label{stenaz}
A_\cL\uphar\cL= \left(P_\cL A^{-1}\uphar\cL\right)^{-1},\; ((A^{-1})_\cL)^{-1}=P_\cL A\uphar\cL.
\]
\end{remark}
\section{Parallel addition of nonnegative selfadjoint linear relations}

\begin{definition}\label{parsumne}
Let $A$ and $B$ be nonnegative selfadjoint l.r. in $\cH$. Then the nonnegative selfadjoint l.r. $A:B$ defined as follows
\[
A:B=\left(A^{-1}\dot+B^{-1}\right)^{-1},
\]
where $\dot+$ is the form sum,
is said to be the parallel sum of $A$ and $B$.
The l.r.
$$ h(A,B):=2(A:B)=\left(\cfrac{A^{-1}\dot+B^{-1}}{2}\right)^{-1}$$
is said to be the harmonic mean of $A$ and $B$.
\end{definition}
Clearly, the parallel sum of two nonnegative selfadjoint l.r. is the nonnegative selfadjoint l.r. as well.
From Definition \ref{parsumne} it follows that
\begin{equation}\label{jhfnyst}
\begin{array}{l}
A^{-1}\dot +B^{-1}=(A:B)^{-1},\;A^{-1}:B^{-1}=(A\dot+B)^{-1},\\
\half (A^{-1}\dot+B^{-1})=(h(A, B))^{-1},\;h(A^{-1}, B^{-1})=\left(\cfrac{A\dot+B}{2}\right)^{-1},
\end{array}
\end{equation}
and
$$\mul(A:B)=\ker(A^{-1}\dot+B^{-1})=\mul A\cap\mul B.$$
Hence, $A:B$ is a selfadjoint operator if and only if $\mul A\cap\mul B=\{0\}.$

The parallel sum $A:B$ is a bounded operator if and only if $\mul A\cap\mul B=\{0\}$ and the quadratic form
$${\mathfrak q}[g]:= A^{-1}[g]+B^{-1}[g],\; g\in\ran A^{\half}\cap\ran B^{\half}$$
is positive definite. In particular, if $A$ or $B$ is a bounded nonnegative selfadjoint operator, then the nonnegative selfadjoint l.r. $A^{-1}\dot+B^{-1}$ has a bounded inverse, hence in this case the l.r. $A:B$ is a bounded nonnegative selfadjoint operator.

The theorem below shows that such way defined parallel addition for nonnegative selfadjoint l.r. preserves basic properties of parallel addition for bounded nonnegative selfadjoint linear operators.
\begin{theorem}\label{Frmsm}
The parallel sum of nonnegative selfadjoint l.r. possesses the properties:
\begin{enumerate}
\item $\lambda A:\lambda B=\lambda (A:B)$, $\lambda>0,$
\item $A:B=B:A,$
\item $(A:B):C=A:(B:C)=\left(A^{-1}\dot+B^{-1}\dot+C^{-1}\right)^{-1},$
\item $A:B\le A,\; A:B\le B,$
\item $A_1\le A_2\Longrightarrow A_1:B\le A_2:B, $
\item $\ran (A:B)^\half=\ran A^{\half}\cap \ran B^{\half},$
\item $A:B=0$ (i.e., $\dom (A:B)=\cH$ and $(A:B)f=0$ $\forall f\in\cH$) if and only if $\ran A^{\half}\cap \ran B^{\half}=\{0\},$
\item if the quadratic form $A^{-1}[\cdot]$ is the closed restriction of the quadratic form $B^{-1}[\cdot]$, then
$$h(A:B)=A.$$
\end{enumerate}
Besides, the following relation
\begin{multline}\label{hinft}
A:B={\rm s-R}-\lim\limits_{n\to\infty}\left(\left(A+\cfrac{1}{n}\, I\right):\left(B+\cfrac{1}{n}\, I\right)\right)\\
={\rm s-R}-\lim\limits_{n\to\infty}\left(\left(A+\cfrac{1}{n}\, I\right)^{-1}+ \left(B+\cfrac{1}{n}\, I\right)^{-1}\right)^{-1}
\end{multline}
is valid.
\end{theorem}
\begin{proof}
The first three properties follows directly from Definition \ref{parsumne}.
Because $A^{-1}\dot+B^{-1}\ge A^{-1}$, we get $A:B=\left(A^{-1}\dot+B^{-1}\right)^{-1}\le A$. Similarly, $A:B\le B$.

$$A:B=0 \Longleftrightarrow\mul(A^{-1}\dot+B^{-1})=\cH\Longleftrightarrow\ran A^{\half}\cap\ran B^{\half}=\{0\}.$$

If $A_1\le A_2$, then $A_1^{-1}\ge A_2^{-1}$ and $A^{-1}_1\dot+B^{-1}\ge A^{-1}_2\dot+B^{-1}$. Hence
$$A_1:B=\left(A^{-1}_1\dot+B^{-1}\right)^{-1}\le \left(A^{-1}_2\dot+B^{-1}\right)^{-1}=A_2:B.$$
Further,
$$\ran (A:B)^\half= \dom \left(A^{-1}\dot+B^{-1}\right)^\half=\dom A^{-\half}\cap \dom B^{-\half}=\ran A^{\half}\cap \ran B^{\half}.$$
If the quadratic form $A^{-1}[\cdot]$ is the closed restriction of the quadratic form $B^{-1}[\cdot]$, then for the form sum $A^{-1}\dot+B^{-1}$ one has that
$A^{-1}\dot+B^{-1}=2 A^{-1}.$
Hence $A:B=\half A$ and $h(A,B)=A$.

For each $n\in\dN$ define the nonnegative selfadjoint l.r. $H_n$ as follows
\[
H_n=\left(\left(A+\cfrac{1}{n}\, I\right)^{-1}+ \left(B+\cfrac{1}{n}\, I\right)^{-1}\right)^{-1}.
\]
Then $\{H_n\}$ is a non-increasing sequence. By \cite[Theorem 3.7]{BHSW2010} there exists
\[
H_\infty:={\rm s-R}-\lim\limits_{n\to\infty}H_n={\rm s-R}-\lim\limits_{n\to\infty}\left(\left(A+\cfrac{1}{n}\, I\right)^{-1}+ \left(B+\cfrac{1}{n}\, I\right)^{-1}\right)^{-1}
\]
and
\[
\begin{array}{l}
\ran H_\infty^\half=\left\{g\in\bigcap\limits_{n=1}^\infty\ran H_n^\half:\lim\limits_{n\to\infty}\left\|(H_n^{-\half})_o g\right\|^2<\infty\right\},\\[3mm]
\left\|(H_\infty^{-\half})_o g\right\|^2=\lim\limits_{n\to\infty}\left\|(H_n^{-\half})_o g\right\|^2.
\end{array}
\]
Since $H_\infty\le H_n$ $\forall n$, we get
\[
H_\infty\le  A+\cfrac{1}{n}\, I,\;H_\infty\le B+\cfrac{1}{n}\, I\qquad\forall n\in\dN.
\]
Hence $H_\infty\le A$ and $H_\infty\le B$.

Note that
\[
H_n^{-1}=\left(A+\cfrac{1}{n}\, I\right)^{-1}+ \left(B+\cfrac{1}{n}\, I\right)^{-1}
\]
is a bounded nonnegative seldfadjoint operator in $\cH$ for each $n\in\dN$ and
\[
\left\|H^{-\half}_n g\right\|^2=\left(\left(A+\cfrac{1}{n}\, I\right)^{-1}g,g\right)+\left(\left(B+\cfrac{1}{n}\, I\right)^{-1}g,g\right),\;g\in\cH.
\]
Besides
\[
\ker H^{-1}_n=\mul A\cap\mul B\;\;\forall n\in\dN.
\]
For each vector $h\in\cH$ the sequence of numbers $ \left\{\left\|H^{-\half}_n g\right\|^2\right\}_{n=1}^\infty$ is non-decreasing. From
\eqref{ranhalf} it follows that
\[
\lim\limits_{n\to\infty}\left\|H^{-\half}_n g\right\|^2<\infty\Longleftrightarrow g\in\ran A^{\half}\cap\ran B^{\half}.
\]
Thus
\[
\begin{array}{l}
\ran H_\infty^\half=\ran A^{\half}\cap\ran B^{\half}=\cD [A^{-1}]\cap \cD [B^{-1}],\\
\left\|H_\infty^{-\half}g\right\|^2=H^{-1}_\infty[g]=A^{-1}[g]+B^{-1}[g]=(A^{-1}\dot+B^{-1})[g],\;g\in \cD [A^{-1}]\cap \cD [B^{-1}].
\end{array}
\]
Therefore, $H^{-1}_\infty=A^{-1}\dot+B^{-1}$, $H_\infty=\left(A^{-1}\dot+B^{-1}\right)^{-1}=A:B$, i.e., \eqref{hinft} holds true.
\end{proof}

\begin{theorem}\label{parsumnesh}
Let $A$ be a nonnegative selfadjoint l.r. in $\cH$ and let $\cL$ be a subspace of $\cH$. Then the following equality
\[
A_\cL={\rm s-R}-\lim\limits_{t\uparrow +\infty}(A:tP_\cL)
\]
holds.
\end{theorem}
\begin{proof}
Set
\[
A(t):=A:tP_\cL,\; t>0.
\]
Then $A(t)\le A$ for all $t>0$, $A(t)$ is a non-decreasing sequences of nonnegative selfadjoint l.r.,
\[
\ran (A(t))^{\half}=\ran A^{\half}\cap \cL\;\forall t>0.
\]
Hence $\ran A(t)\subseteq \cL$ for all $t>0$ and, therefore, see \eqref{setxi}, $A(t)\in\Xi(A,\cL)$ for all $t>0$.
From \cite[Theorem 3.1]{BHSW2010} it follows that there exists
\begin{equation}\label{anul}
{\rm s-R}-\lim\limits_{t\uparrow +\infty} A(t)=:A_0.
\end{equation}
Since $\ker A(t)\supseteq\cL^\perp$, we get $\ker A_0\supseteq\cL^\perp\Longleftrightarrow \ran A_0\subseteq\cL$. Hence $A_0 \in\Xi(A,\cL)$.

Let $\wt A\in\Xi(A,\cL)$.
Set
\[
\wt H_n(t):=\left(\left(\wt A+\cfrac{1}{n}\, I\right)^{-1}+ \left(tP+\cfrac{1}{n}\, I\right)^{-1}\right)^{-1},\; n\in\dN,\; t>0.
\]
Because $\ker \wt A\supseteq\cL^\perp$ and $\ker P_\cL=\cL^\perp,$ $P_\cL\uphar\cL=I_\cL$,
we have that
\[
\wt H_n(t)\uphar\cL^\perp=\cfrac{1}{2n} I_{\cL^\perp},\; (\wt H_n(t))^{-1}\uphar\cL=\left(\wt A+\cfrac{1}{n}\, I_\cL\right)^{-1}\uphar\cL+\left(t+\cfrac{1}{n}\right)^{-1}\uphar\cL
\]
for all $n\in\dN$ and for all $t>0.$
Hence
\[
{\rm s-R}-\lim\limits_{n\to\infty} (\wt H_n(t))^{-1}\uphar\cL=\wt A^{-1}\uphar\cL+t^{-1}I_\cL\;\;\forall t>0.
\]
Therefore,
\[
\wt A:tP_\cL={\rm s-R}-\lim\limits_{n\to\infty} \wt H_n(t)=\wt A.
\]
Now the inequality $A\ge\wt A$ for all $\wt A\in \Xi(A,\cL)$ implies $A:tP_\cL\ge \wt A: tP_\cL$ for all $t>0$ and letting $t$ to  $+\infty$ from \eqref{anul} we get
\[
A_0\ge \wt A\;\;\forall\wt A\in \Xi(A,\cL).
\]
Due to Theorem \ref{shortrel}, the shorted l.r. $A_\cL$ is the maximal element of  $\Xi(A,\cL)$. Therefore, $A_0=A_\cL.$ The proof is complete.
\end{proof}

\begin{corollary}\label{afqyj}
Let $A$ and $B$ be nonnegative selfadjoint l.r. and let $\cL$ be a subspace.
Then
\[
(A:B)_\cL=A_\cL:B=A:B_\cL=A_\cL:B_\cL.
\]
\end{corollary}
\begin{proof}
By definition
\[
(A:B)_\cL=\max\{C=C^*, C\ge 0:C\le A:B,\ran C\subset \cL\}
\]
Since $A_\cL:B\le A:B$ and
$$\ran (A_\cL:B)^{\half}=\ran (A_\cL)^{\half}\cap\ran B^{\half}=\cL\cap\ran A^{\half}\cap \ran B^{\half}\subset\cL, $$
we get that $(A:B)_\cL\ge A_\cL:B$ and similarly $(A:B)_\cL\ge A:B_\cL.$

On the other side because $B:tP_\cL$ is a non-decreasing function w.r.t. $t> 0$ and $B:tP_\cL\le B_\cL$, one has
\[
(A:B)_\cL={\rm s-R}-\lim\limits_{t\uparrow+\infty}((A:B):tP_\cL)={\rm s-R}-\lim\limits_{t\uparrow+\infty}(A:(B:tP_\cL))\le A:B_\cL
\]
Hence
\[
A:B_\cL\ge (A:B)_\cL\ge A:B_\cL\Longleftrightarrow (A:B)_\cL=A:B_\cL
\]
and similarly $(A:B)_\cL=A_\cL:B$.

Since $A_\cL\le A$, we get $A_\cL:B_\cL\le A:B_\cL=(A:B)_\cL$.
On the other side
\[
A_\cL\ge A:tP_\cL\Longrightarrow A_\cL:B_\cL\ge (A:tP_\cL):B_\cL=(A:B_\cL):tP_\cL \;\;\forall t>0
\]
Hence
\[
A_\cL:B_\cL\ge {\rm s-R}-\lim\limits_{t\uparrow+\infty}\left((A:B_\cL):tP_\cL\right)=(A:B_\cL)_\cL=A:B_\cL=(A:B)_\cL.
\]

\end{proof}

\section{The inequality between the harmonic and arithmetic means of nonnegative selfadjoint linear relations}\label{neravfa}

If $a$ and $b$ are positive numbers, then the harmonic mean and the arithmetic mean of $a$ and $b$ are connected by the inequality
\[
h(a,b)=\left(\cfrac{a^{-1}+b^{-1}}{2}\right)^{-1}\le \cfrac{a+b}{2},
\]
and the equality holds if and only if $a=b$.

Let $A,B\in\bB^+(\cH)$ and let $\sL:=\cran (A+B)$, then since $A+B\ge A, B$, there exists a nonnegative selfadjoint contraction $K$ in $\sL$ such that
\[
A=(A+B)^\half K(A+B)^\half,\; B=(A+B)^\half (I_\sL-K)(A+B)^\half.
\]
Actually, since $||(A+B)^\half f||^2\ge ||A^\half f||^2$ for all $f\in\cH$, by the Douglas factorization theorem \cite{Doug} there exists a contraction $X\in\bB(\cH,\sL)$ such that $A^\half =(A+B)^\half X$. Then $A=(A+B)^\half XX^*(A+B)^\half$.
The following equality is established in \cite{Ar2}:
\begin{equation}\label{yanash}
h(A,B)=2(A+B)^\half(I_\sL-K)K(A+B)^\half.
\end{equation}
It follows that
\[
h(A,B)\le \cfrac{A+B}{2}
\]
and the equality holds if and only if $A=B$.

The goal of this section is to prove a similar inequality for the case of nonnegative selfadjoint l.r..
First, we find the Cayley transforms of the arithmetic and harmonic means.
\begin{lemma}\label{dcgv1}
Let  $M$  be a nonnegative selfadjoint contraction in the Hilbert space $\sL$.
Then
\[
(M-M^2)^{-1}[f]=M^{-1}[f]+(I-M)^{-1}[f]\;\;\forall f\in\ran (M-M^2)^\half.
\]
\end{lemma}
\begin{proof}
It is sufficient to consider the case $\ker M=\ker(I-M)=\{0\}$. Then for $f=(M-M^2)^\half u$ we get
\[
M^{-\half}f=(I-M)^\half u,\;
(I-M)^{-\half}f=M^\half u,\; (M-M^2)^{-\half}f=u.
\]
Hence
\[
\begin{array}{l}
||(M-M^2)^{-\half}f||^2=||u||^2,\; ||M^{-\half}f||^2=||(I-M)^\half u||^2=||u||^2-||M^\half u||^2,\\
||(I-M)^{-\half}f||^2=||M^\half u||^2.
\end{array}
\]
This completes the  proof.
\end{proof}
\begin{theorem}\label{vspom2}
Let $A$ and $B$ be nonnegative selfadjoint l.r. and let
$$T_A:={\mathfrak C}(A),\; T_B:={\mathfrak C}(B)$$
be the Cayley transforms of $A$ and $B$, respectively.
Then
\begin{equation}\label{EQ11}
{\mathfrak C}\left(\half(A\dot +B)\right)= h(I+T_A,I+T_B)-I,
\end{equation}
\begin{equation}\label{EQ22}
{\mathfrak C}\left(h(A,B)\right)= I-h(I-T_A,I-T_B).
\end{equation}

\end{theorem}
\begin{proof}
Set
\[
\begin{array}{l}
\wt S:=\half(A\dot +B),\;L_A=I+T_A, L_B=I+T_B,\\[3mm]
 \wt T=h(I+T_A,I+T_B)-I, \;  \wt L=I+\wt T.
\end{array}
\]
Since
\[
\wt T\ge -I,\;\wt T\le\half(I+T_A+I+T_B)-I=\half(T_A+T_B)\le I,
\]
the operator $\wt T$ is a selfadjoint contraction.
Using \eqref{yanash} we get the representation
\begin{equation}
\label{repfur}
h(I+T_A,I+T_B)=I+\wt L=2(L_A+L_B)^{\half}(M-M^2)(L_A+L_B)^{\half},
\end{equation}
where $M$ is a nonnegative selfadjoint contraction acting in the subspace $\cran(L_A+L_B)$ and such
that
 \begin{equation}\label{jghfnj1}
L_A=(L_A+L_B)^{\half}M(L_A+L_B)^{\half}, \;L_B=(L_A+L_B)^{\half}(I-M)(L_A+L_B)^{\half}.
\end{equation}

From \eqref{AJHVF} we have
\[
\cD[A]=\ran L^\half_A,\; A[v]+||v||^2=2L^{-1}_A[v],\;v\in \cD[A],
\]
\[
\cD[B]=\ran L^\half_B,\; B[u]+||u||^2=2L^{-1}_B[u],\;u\in \cD[B],
\]
\[
\cD[\wt S]=\ran L^\half_A\cap\ran L^\half_B,\; \wt S[f]+||f||^2=L^{-1}_A[f]+L^{-1}_B[f],\;f\in \cD[\wt S].
\]
From \eqref{jghfnj1} one gets
\[
\begin{array}{l}
\left\{\begin{array}{l}\ran L^\half_A=(L_A+L_B)^\half\ran M^{\half},\\
L^{-1}_A[v]=\left\|M^{[-\half]}(L_A+L_B)^{[-\half]}v\right\|^2,\;\; v\in \ran L^\half_A\end{array}\right.,\\[3mm]
\left\{\begin{array}{l} \ran L^\half_B=(L_A+L_B)^\half\ran(I- M)^{\half},\\
L^{-1}_B[u]=\left\|(I-M)^{[-\half]}(L_A+L_B)^{[-\half]}u\right\|^2,\;\; u\in \ran L^\half_B\end{array}\right..
\end{array}
\]
Hence
\[
\left\{\begin{array}{l} \cD[\wt S]=(L_A+L_B)^\half\left(\ran M^\half\cap\ran (I-M)^\half\right)= (L_A+L_B)^\half\left(\ran(M-M^2)^\half\right),\\
\wt S[f]+||f||^2=\left\|M^{[-\half]}(L_A+L_B)^{[-\half]}f\right\|^2+\left\|(I-M)^{[-\half]}(L_A+L_B)^{[-\half]}f\right\|^2,\;f\in \cD[\wt S]\\
\end{array}\right..
\]
Now Lemma \ref{dcgv1} and \eqref{repfur} yield that
\[
\left\{\begin{array}{l} \cD[\wt S]=\ran (I+\wt L)^{\half},\\
\wt S[f]+||f||^2=(M-M^2)^{-1}[(L_A+L_B)^{[-\half]}f]=2||(I+\wt L)^{[-\half]}f||^2,\;f\in \cD[\wt S]
\end{array}\right..
\]
Taking into account \eqref{AJHVF} we obtain that ${\mathfrak C}(\wt S)=\wt L$.

Since
\[
{\mathfrak C}(A^{-1})=-T_A,\; {\mathfrak C}(B^{-1})=-T_B,
\]
we get
\[
{\mathfrak C}\left(\half(A^{-1}\dot+B^{-1})\right)=h(I-T_A, I-T_B)-I.
\]
Then
\[
{\mathfrak C}\left(h(A,B)\right)={\mathfrak C}\left(\left(\cfrac{A^{-1}\dot +B^{-1}}{2}\right)^{-1}\right)= I-h(I-T_A,I-T_B).
\]
\end{proof}
Now from \eqref{rtkbtr}, \eqref{AJHVF} and \eqref{EQ22} we get the equalities
\[
h(A,B)=\left\{\left\{\left(2I-h(I-\sC(A), I-\sC(B))\right)f,\left(h(I-\sC(A), I-\sC(B))\right)f\right\},\; f\in\cH\right\}
\]
\begin{equation}\label{ctqxfc}
\cD[A:B]=\cD[h(A,B)]=\ran\left(2I-h(I-\sC(A), I-\sC(B))\right)^{\half}.\\[3mm]
\end{equation}
The next statement is an analogue of \cite[Theorem 2.5]{PSh}.
\begin{corollary}\label{cylbv}
Let $\{A_n\}$ and $\{B_n\}$ be two sequences of nonnegative selfadjoint l.r.. Then in the strong resolvent convergence sense
$$A_n\nearrow A_\infty,\; B_n\nearrow B_\infty\Longrightarrow A_n\dot+B_n\nearrow A_\infty\dot+B_\infty,$$
$$A_n\searrow A_\infty,\;B_n\searrow B_\infty\Longrightarrow(A_n:B_n)\searrow (A_\infty:B_\infty).$$
\end{corollary}
\begin{proof}
Set $T_{A_n}:=\mathfrak C(A_n),$ $T_{B_n}:=\mathfrak C(B_n),\;n\in\dN,$ $T_{A_\infty}:=\mathfrak C(A_\infty),$ $T_{B_\infty}:=\mathfrak C(B_\infty).$

If $A_n\nearrow A_\infty$ and $B_n\nearrow B_\infty$ in the strong resolvent sense, then $T_{A_n}\searrow T_{A_\infty}$ and $T_{B_n}\searrow T_{B_\infty}$ in the strong sense.
It follows from \eqref{EQ11} and \cite[Theorem 2.5.]{PSh} that
\begin{multline*}
\mathfrak C\left(\half(A_n\dot+B_n)\right)=h(I+T_{A_n}, I+T_{B_n})-I\\=2( I+T_{A_n}: I+T_{B_n})-I\searrow 2(I+T_{A_\infty}:I+T_{B_\infty}) -I
 =h(I+T_{A_\infty},I+T_{B_\infty})-I.
\end{multline*}
Now
\[
\half(A_\infty\dot+B_\infty)=\mathfrak C\left(h(I+T_{A_\infty},I+T_{B_\infty})-I\right)={\rm s-R}-\lim\limits_{n\to \infty}\half(A_n\dot+B_n).
\]
Similarly, if $A_n\searrow A_\infty$ and $B_n\searrow B_\infty$, then, taking into account equality \ref{EQ22} and \cite[Theorem 2.5.]{PSh} once again, we get that $(A_n:B_n)\searrow (A_\infty:B_\infty).$
\end{proof}

For nonnegative selfadjoint l.r. $A$ and $B$ let us define a nonnegative selfadjoint l.r. $c_0(A,B)$ as follows
\begin{multline}\label{chtlytt}
c_0(A,B):=\mathfrak C\left(\frac{{\mathfrak C}(A)+ {\mathfrak C}(B)}{2}\right)\\
=\left\{\left\{\left(I+\frac{{\mathfrak C}(A)+ {\mathfrak C}(B)}{2}\right)f,\left(I-\frac{{\mathfrak C}(A)+ {\mathfrak C}(B)}{2}\right)f\right\}:\; f\in\cH\right\}.
\end{multline}
If $a$ and $b$ are positive numbers, then
\begin{equation}\label{chtlyttt}
c_0(a,b)=\cfrac{a+b+2ab}{2+a+b}.
\end{equation}
Observe that the following inequalities are valid
\[
\cfrac{2ab}{a+b}\le \cfrac{a+b+2ab}{2+a+b}\le \cfrac{a+b}{2},
\]
and the sign $"="$ holds if and only if $a=b.$
\begin{theorem}\label{ythdj}
Let $A$ and $B$ be a nonnegative selfadjoint l.r.. Then
\begin{enumerate}
\item $\cD[c_0(A,B)]=\cD[A]+\cD[B];$
\item $\ran \left(c_0(A,B)\right)^\half=\ran A^\half+\ran B^\half;$
\item the inclusion
\begin{equation}\label{hjhji}
\cD[A:B]\supseteq\cD[A]+\cD[B]
\end{equation}
holds and the inequality
\[
h(A,B)\le c_0(A,B)\le \cfrac{A\dot+B}{2}
\]
is valid; moreover, the sign $"="$ holds if and only if $A=B$.
\end{enumerate}
In addidion
\begin{equation}\label{hjhji2}
\ran(A\dot+B)^{\half}\supseteq\ran A^\half+\ran B^\half.
\end{equation}
\end{theorem}
\begin{proof}
Set
\[
T_A:={\mathfrak C}(A),\;T_B:={\mathfrak C}(B).
\]
Because $I\pm T_A$ and $I\pm T_B$ are bounded nonnegative operators, we have the inequalities
\[
h(I\pm T_A, I\pm T_B)\le \half(I\pm T_A+I\pm T_B)=I\pm \half(T_A+T_B).
\]
Using \eqref{EQ11}, \eqref{EQ22}, and \eqref{chtlytt} we get
\begin{multline*}
h(I+T_A, I+T_B)-I\le\half(T_A+T_B)\le I-h(I-T_A, I-T_B)\\
\Longleftrightarrow \mathfrak C\left(\half(A\dot+B)\right)\le \mathfrak C\left(c_0(A,B)\right)\le \mathfrak C(h(A,B))
\Longleftrightarrow h(A,B)\le c_0(A,B)\le \half(A\dot+B).
\end{multline*}

If $A=B$, then, clearly,
\[
c_0(A,B)=A,\;\half(A\dot+B)=A,\; \half(A^{-1}\dot+B^{-1})=A^{-1}\Longrightarrow h(A,B)=A.
\]
Thus, $h(A,B)=c_0(A,B)=\half(A\dot+B).$

Suppose that $h(A,B)=\half(A\dot+B)$ for some nonnegative selfadjoint l.r. $A$ and $B$. Then
\[
I-h(I-T_A, I-T_B)=h(I+T_A, I+T_B)-I.
\]
Since
\[
\half(T_A+T_B)\le I-h(I-T_A, I-T_B)  =h(I+T_A, I+T_B)-I \le\half(T_A+T_B),
\]
we get
\[
h(I+T_A, I+T_B)-I=\half(T_A+T_B)\Longleftrightarrow h(I+ T_A, I+ T_B)= \half(I+ T_A+I+ T_B).
\]
Because $I+T_A$ and $I+T_B$ are bounded, the latter equality yields the equality $T_A=T_B$ and hence
$$A={\mathfrak C}(T_A)={\mathfrak C}(T_B)=B.$$
Similarly, one can prove that $h(A,B)=c_0(A,B)\Longrightarrow A=B$, and $c_0(A,B)=\half(A\dot+B)\Longrightarrow A=B.$

From \eqref{AJHVF} and \eqref{EQ22} we have
\[
\begin{array}{l}
\cD[A:B]=\cD[h(A,B)]=\ran\left(2I-h(I-T_A,I-T_B)\right)^{\half},\\[3mm]
\cD[c_0(A,B)]=\ran\left(\cfrac{I+T_A}{2} +\cfrac{I+T_B}{2}\right)^{\half},\;
\ran\left(c_0(A,B)\right)^\half=\ran \left(\cfrac{I-T_A}{2} +\cfrac{I-T_B}{2}\right)^{\half}.
\end{array}
\]
Since
\[
h(I\pm T_A,I\pm T_B)\le\cfrac{(I\pm T_A)+(I\pm T_B)}{2}=I\pm\cfrac{T_A+T_B}{2},
\]
we obtain the inequality
\[
2I-h(I\pm T_A,I\pm T_B)\ge\cfrac{I\mp T_A}{2} +\cfrac{I\mp T_B}{2}.
\]
Now the Douglas theorem \cite{Doug} yields the inclusions
\[
\ran\left(2I-h(I\pm T_A,I\pm T_B)\right)^{\half}\supseteq\ran\left(\cfrac{I\mp T_A}{2} +\cfrac{I\mp T_B}{2}\right)^{\half}.
\]
But, see \cite[Theorem 2.2]{FW},
\[
\ran\left(\cfrac{I\pm T_A}{2} +\cfrac{I\pm T_B}{2}\right)^{\half}=\ran (I\pm T_A)^{\half}+\ran (I\pm T_B)^{\half}.
\]
Using \eqref{ctqxfc} and equalities
\begin{multline*}
\cD[A]=\ran (I+T_A)^{\half},\;\cD[B]=\ran (I+T_B)^{\half},\\
\ran A^\half=\ran (I-T_A)^\half,\;\ran B^\half=\ran (I-T_A)^\half
\end{multline*}
we get  the equalities
$$\cD[c_0(A,B)]=\cD[A]+\cD[B],\;\ran\left (c_0(A,B)\right)^\half =\ran A^\half+\ran B^\half,$$
and
finally arrive at \eqref{hjhji}.
Since
\[
\ran\left(2I-h(I+T_A,I+T_B)\right)^\half\supseteq \ran\left((I-T_A)+(I-T_B)\right)^{\half},
\]
we get \eqref{hjhji2}.

The proof is complete.

\end{proof}
As a consequence of \eqref{hjhji} we note, that
if one of two nonnegative selfadjoint l.r. is a linear operator, then their parallel sum is a nonnegative selfadjoint linear operator as well, moreover, if one of two
is a bounded operator, then the parallel sum is a bounded operator.

\section{Arithmetic--harmonic means} \label{fhbaufh}
We consider analogs of the arithmetic-harmonic means for the case of nonnegative selfadjoint l.r..
\begin{theorem}\label{arharm}
Let $A$ and $B$ be nonnegative selfadjoint l.r.. 
Define
\begin{equation}\label{fhbauf1}
A_0:=A,\; B_0:=B,\; A_{n}:=\half(A_{n-1}\dot+B_{n-1}),\; B_{n}:=h(A_{n-1},B_{n-1}),\; n\in\dN.
\end{equation}
Then the sequence $\{A_n\}$ is a non-increasing, while the sequence $\{B_n\}$ is a non-decreasing and $A_n\ge B_n$ for all $n\in\dN$.
Moreover, the nonnegative selfadjoint l.r.
\begin{equation}\label{harar1}
A_{\infty}:={\rm s-R}-\lim\limits_{n\to\infty}A_n,\;B_{\infty}:={\rm s-R}-\lim\limits_{n\to\infty}B_n
\end{equation}
posses the following properties:
\[
\begin{array}{l}
h(A,B)\le B_\infty\le A_\infty\le \half(A\dot+B),\\[2mm]
\cD[B_\infty]\supseteq \cD[A_\infty]\supseteq\cD[A]\cap\cD[B],\\[2mm]
\ran A^{\half}_\infty\supseteq\ran B^\half_\infty\supseteq\ran A^\half\cap\ran B^{\half},\\[2mm]
A_\infty[\f]=B_\infty[\f]\;\;\forall \f\in\cD[A_\infty],\\[2mm]
A^{-1}_\infty[\psi]=B^{-1}_\infty[\psi]\;\;\forall \psi\in\ran B^{\half}_\infty.
\end{array}
\]
If $A^{-1}$ and $B^{-1}$ are bounded (i.e. they are graphs of bounded nonnegative selfadjoint operators), then
\[
A_\infty=B_\infty=(A^{-1}\#B^{-1})^{-1}.
\]

\end{theorem}
\begin{proof}
By induction from the definitions of form sums, harmonic means and from Theorem \ref{ythdj} one can show that
\[
\begin{array}{l}
\cD[A_n]=\cD[A]\cap\cD[B],\; \cD[B_n]\supseteq\cD[A]+\cD[B],\\
 h(A,B)=B_1\le B_{n}\le B_{n+1}\le A_{n+1}\le A_{n}\le A_1=\half(A\dot+B)\;\;\forall n\in\dN.
 \end{array}
\]
Since the sequence $\{B_n\}$ is a non-decreasing and the sequence $\{A_n\}$ is a non-increasing, from \cite[Theorem VII.3.11]{Ka}, \cite[Theorem 3.1, Theorem 3.2]{Simon1978}, \cite[Theorem 4.2, Theorem 4.3]{BHSW2010} it follows that the strong resolvent limits of $\{A_n\}$ and $\{B_n\}$ exist and
\[
h(A,B)\le B_\infty\le A_\infty\le \half(A\dot+B),\;\cD[B_\infty]\supseteq \cD[A_\infty].
\]
Since
\[
\cD[B_\infty]=\left\{h\in\bigcap\limits_{n=1}^\infty\cD[B_n]:\lim\limits_{n\to\infty}B_n[h]<\infty\right\}
\]
and $B_n\le A_1$ for all $n\in\dN$, $\bigcap\limits_{n=1}^\infty \cD[B_n]\supseteq\cD[A]\cap\cD[B]=\cD[A_1]$, we get that
$$\cD[B_\infty]\supseteq\cD[A]\cap\cD[B].$$

Fix $n\in\dN$. Then for any $f\in\cD[A]\cap \cD[B]$ by definition  and from $B_{n+1}\ge B_n$ we have
\begin{multline*}
A_{n+1}[f]-B_{n+1}[f]=\half A_n[f]+\half B_{n}[f]-B_{n+1}[f]\\
\le \half A_n[f]+\half B_{n}[f]-B_n[f]\le\cfrac{A_n[f]-B_n[f]}{2} .
\end{multline*}
Hence
\[
0\le A_{n+1}[f]-B_{n+1}[f]\le \cfrac{A_1[f]-B_1[f]}{2^n}\;\;\forall f\in\cD[A]\cap \cD[B],\;\forall n\in\dN.
\]
It follows that
\begin{equation}\label{novlim}
\lim\limits_{n\to \infty}A_n[f]=\lim\limits_{n\to\infty}B_n[f]=B_\infty[f]\;\;\forall f\in\cD[A]\cap\cD[B].
\end{equation}
Define the sesquilinear form $\mathfrak t$ as follows
\[
\dom\mathfrak t=\bigcup\limits_{n=1}^\infty\cD[A_n],\; \mathfrak t[f,g]=\lim\limits_{n\to\infty}A_n[f,g],\; f,g\in\dom \mathfrak t.
\]
Since $\cD[A_n]=\cD[A]\cap\cD[B]$, we have the equality $\dom\mathfrak t=\cD[A]\cap\cD[B]$ and from \eqref{novlim}:
\[
\mathfrak t[f,g]=B_\infty[f,g]\;\;\forall f,g,\in\cD[A]\cap\cD[B].
\]
Since the form $B_\infty[\cdot,\cdot]$ restricted on $\cD[A]\cap\cD[B]$ is closable, using \cite[Theorem 3.2]{Simon1978}, \cite[ Theorem 4.3]{BHSW2010},
we obtain that the closure of the form $\mathfrak t[\cdot,\cdot]$ coincides with $A_\infty[\cdot,\cdot]$. Hence $A_\infty[\f]=B_\infty[\f]$ for all $\f\in \cD[A_\infty].$

Since
\[
\begin{array}{l}
A^{-1}_n=2\left(A_{n-1}\dot+B_{n-1}\right)^{-1}=h(A^{-1}_{n-1}, B^{-1}_{n-1}),\\[3mm]
B^{-1}_n=\left(h(A_{n-1}, B_{n-1})\right)^{-1}=\half(A^{-1}_{n-1}\dot+B^{-1}_{n-1}),\; n\in\dN,
\end{array}
\]
the sequences $\{A^{-1}_n\},\{B^{-1}_n\}\subset\bB^{+}(\cH)$ are the non-decreasing and non-increasing, respectively, and
\[
{\rm s-R}-\lim\limits_{n\to\infty}A^{-1}_n=A^{-1}_{\infty},\;{\rm s-R}-\lim\limits_{n\to\infty}B^{-1}_n=B^{-1}_{\infty}.
\]
Arguing as above we see that
$$ B^{-1}_\infty\ge A^{-1}_\infty,\;\ran A^{\half}_\infty\supseteq \ran B^{\half}_\infty\supseteq\ran A^{\half}\cap\ran B^{\half},$$
and
$$A^{-1}_\infty[\psi]=B^{-1}_\infty[\psi]\;\;\forall \psi\in \ran B^\half_\infty.$$

Suppose that $A^{-1}_0=A^{-1}$ and $B^{-1}_0=B^{-1}$ are bounded. Then from \eqref{jhfnyst} 
it follows that $\{A^{-1}_n\},\{B^{-1}_n\}\subset\bB^+(\cH)$, and there is a common limit
\[
{\rm s}-\lim\limits_{n\to\infty}A^{-1}_n={\rm s}-\lim\limits_{n\to\infty}B^{-1}_n=A^{-1}\#B^{-1}.
\]
Due to the equalities
\[
A_\infty=\left({\rm s}-\lim\limits_{n\to\infty}A^{-1}_n\right)^{-1},\;B_\infty=\left({\rm s}-\lim\limits_{n\to\infty}B^{-1}_n\right)^{-1},
\]
we get $A_\infty=B_\infty=\left(A^{-1}\#B^{-1}\right)^{-1}.$
The proof is complete.
\end{proof}

In the sequel we will denote by $ah(A,B)$ the arithmetic--harmonic means of nonnegative selfadjoint l.r. $A$ and $B$. By Theorem \ref{arharm} we have
\[
ah(A,B)=\left<A_\infty, B_\infty\right>,
\]
where $A_\infty$ and $B_\infty$ are defined in \eqref{fhbauf1} and \eqref{harar1}.
Then from \eqref{jhfnyst}
\[
ah(A^{-1}, B^{-1})=\left<B^{-1}_\infty, A^{-1}_\infty\right>.
\]

The next corollaries, show that in general $A_\infty\ne B_\infty$.
\begin{corollary}\label{dikij}
Let $B$ be a nonnegative selfadjoint l.r.. Define
\[
\begin{array}{l}
\cB':=\{\cdom B\oplus\{0\}\}\bigoplus\{\{0\}\oplus\mul B\}=\left\{\{\f, \psi\}:\f\in\cdom B,\psi\in\mul B\right\},\\
\cB'':=\{\ker B\oplus\{0\}\}\bigoplus\{\{0\}\oplus\cran B\}=\left\{\{f, g\}:f\in\ker B,g\in\cran B\right\}.
\end{array}
\]
Then
\begin{enumerate}
\item for $A=0$ ($\dom A=\cH$) one has $ah(A,B)=\left<\cB',A\right>,$
\item for $A=\{0\}\oplus\cH$ one has $ah(A,B)=\left<A, \cB''\right>.$
\end{enumerate}
\end{corollary}
\begin{proof}
(1) If $A$ is zero operator, defined on $\cH$, then for sequences $\{A_n\}$ and $\{B_n\}$, defined in \eqref{fhbauf1}, we have $A_1:=\half(A\dot+B)=\half B$, $B_1:=h(A,B) =A,$
\[
A_n=\cfrac{1}{2^n}B,\; B_n=A,\; n\in\dN.
\]
Thus $A_\infty=\cB'$, $B_\infty=A$.

(2) If $A=\{0\}\oplus\cH$, then $A^{-1}$ is the zero operator defined on $\cH$. Hence we can use \eqref{jhfnyst} and the arguments above.
\end{proof}
\begin{corollary}\label{dikij2}
Let $A$ be a nonnegative selfadjoint l.r.. Set
\begin{equation}\label{bgh1}
\sM:=\mul A\oplus\ker A.
\end{equation}
Then $ah(A,A^{-1})=\left< \left(P_{\sM^\perp}\right)^{-1}, P_{\sM^\perp}\right>$.
In particular $ah(A,A^{-1})=I\Longleftrightarrow\sM=\{0\}.$

\end{corollary}
\begin{proof}
Set $T:=\mathfrak C(A)$, then $\mathfrak C(A^{-1})=-T,$ 
\[
A_1:=\half(A\dot+A^{-1}),\; B_1:=h(A,A^{-1})=\left(\half(A^{-1}\dot+A)\right)^{-1}=A^{-1}_1.
\]
By induction
\[
A_n:=\half(A_{n-1}\dot+B_{n-1})=\half(A_{n-1}\dot+A^{-1}_{n-1}),\; B_n:=h(A_{n-1},B_{n-1})=A^{-1}_n,\; n\ge 2.
\]
Using \eqref{EQ11} and \eqref{EQ22} we get
\[
\begin{array}{l}
\mathfrak C(A_1)=h(I+T, I-T)-I=2(I+T)(I-T)(I+T+I-T)^{-1}-I=-T^2,\\
 \mathfrak C(B_1)=\mathfrak C(A^{-1}_1)=-\mathfrak C(A_1)=T^2,\;
\mathfrak C(A_n)=-T^{2^n},\; \mathfrak C(B_n)=T^{2^n},\; n\ge 2.
\end{array}
\]
Since $T$ is a selfadjoint contraction and $\sM=\ker (I-T^2)$, where $\sM$ is defined by \eqref{bgh1},
we get the equality
\[
{\rm s}-\lim\limits_{n\to \infty}T^{2^n}=P_\sM.
\]
It follows that
\[
B_\infty=\mathfrak C(P_\sM)=(I-P_\sM)(I+P_\sM)^{-1}=P_{\sM^\perp},\;A_\infty=\mathfrak C(-P_\sM)=\left(P_{\sM^\perp}\right)^{-1}.
\]
The proof is complete.
\end{proof}
\section{Arithmetic, harmonic, arithmetic--harmonic means and nonnegative selfadjoint extensions of nonnegative symmetric linear relations}\label{EXTEN}
Basic results  of the Kre\u{\i}n theory \cite{Kr} can be extended to the case of a non-densely defined  symmetric operator or a symmetric l.r.
\cite{AN, Ar4, CS,HMS,HSSW06, HSSW07}. Let $S$ be a nonnegative symmetric l.r. in $\cH$.
Then the set of all its nonnegative selfadjoint extensions consists of maximal and minimal elements. The maximal nonnegative selfadjoint extension is the Friedrichs extension
$S_{\rm F}$  \cite{RoBe}, which is associated with the closure of the sesquilinear form
$${\mathfrak s}[f,g]=(f', g),\; \{f,f'\}, \{g,g'\}\in S .$$
 The minimal nonnegative selfadjoint extension $S_{\rm K}$ is called the Kre\u{\i}n or the Kre\u{\i}n-von Neumann extension and can be defined as follows \cite{CS}:
\[
S_{\rm K}=((S^{-1})_{\rm F})^{-1}.
\]
The closed form $S_{\rm F}[\cdot,\cdot]$ is a closed restriction of the closed form $\wt S[\cdot,\cdot]$ associated with any nonnegative selfadjoint extension $\wt S$ of $S$ \cite{Kr,Ar4}.
A nonnegative selfadjoint extension $\wt S$ of $S$ is called \textit{extremal} \cite{Ar4} if
\[
\inf\left\{(f'-g', f-g):\{g,g'\}\in S\right\}=0 \quad\mbox{for all} \quad \{f,f'\}\in\wt S.
\]
It is proved in \cite[Proposition 3]{Ar4} that if $\wt S$ is an extremal extension of $S$, then the closed form $\wt S[\cdot,\cdot]$ is a closed restriction of the closed form $S_{\rm K}[\cdot,\cdot]$ associated with the Kre\u{\i}n extension of $S$.

Let $Q={{\mathfrak C}}(S)$ be the Cayley transform of $S$. Then $Q$ is a non-densely defined Hermitian contraction with $\dom Q=\ran (S+I)$. There is a one-to-one correspondence \cite{Kr,CS}
\[
\wt Q={{\mathfrak C}}(\wt S),\;\wt S={{\mathfrak C}}(\wt Q)
\]
between the set of all nonnegative selfadjoint extensions and the set of all selfadjoint contractive extensions of $Q$.
Let
\[
Q_\mu={{\mathfrak C}}(S_{\rm F}),\; Q_M={{\mathfrak C}}(S_{\rm K}).
\]
The operators $Q_\mu$ and $Q_M$ posses the following properties \cite{Kr}:

if $\sN:=\cH\ominus\dom Q,$ then
\begin{equation}\label{rhfqybt}
(I+Q_\mu)_\sN=(I-Q_M)_\sN=0.
\end{equation}
$\sN$ is the deficiency subspace of $S$ corresponding to $\lambda=-1$. Set
$\sN_0:=\cran(Q_M-Q_\mu).$
The set of all selfadjoint contractive extensions of $Q$ forms the operator interval $[Q_\mu, Q_M]$ \cite{Kr}, which admits the parameterizations
\begin{multline}\label{krparm}
\wt Q=\half(Q_\mu+Q_M)+\half(Q_M-Q_\mu)^\half \wt Z(Q_M-Q_\mu)^{\half}\\=Q_\mu+\half(Q_M-Q_\mu)^\half(I_{\sN_0}+ \wt Z)(Q_M-Q_\mu)^{\half}
=Q_M-\half(Q_M-Q_\mu)^\half(I_{\sN_0}- \wt Z)(Q_M-Q_\mu)^{\half}.
\end{multline}
A nonnegative selfadjoint extension $\wt S$ is extremal if and only if the parameter $\wt Z$ for $\wt Q={{\mathfrak C}}(\wt S)$ in the right hand side in \eqref{krparm} is a selfadjoint and unitary operator (a fundamental symmetry) in the subspace $\sN_0$ \cite{ArlBelTsek2011}. This is equivalent to the equality $(I-\wt Q^2)_\sN=0$, see \cite{ArlBelTsek2011}.
If $\dim \sN_0=1$, then extremal extensions of $S$ are only $S_{\rm F}$ and $S_{\rm K}$.

From \eqref{krparm} and \eqref{rhfqybt} it follows the equalities
\begin{equation}\label{shoort}
(I\pm\wt Q)_\sN=\half(Q_M-Q_\mu)^\half(I_{\sN_0}\pm\wt Z)(Q_M-Q_\mu)^\half.
\end{equation}
Note that from \eqref{rhfqybt} and \eqref{gthda1}, \eqref{gthda2} one gets the equalities for shorted operators
\[
(S_{\rm K})_\sN=0,\; ((S_{\rm F})^{-1})_\sN=0.
\]
Moreover, the Kre\u{\i}n uniqueness criteria \cite{Kr} is equivalent to the equality $(S_{\rm F})_\sN=0$.

The next theorem shows connections of the arithmetic and harmonic means with the theory of nonnegative selfadjoint extensions of nonnegative symmetric l.r..
\begin{theorem}\label{ext}
Let $S$ be a nonnegative symmetric l.r.. Then
\begin{enumerate}
\item
for an arbitrary nonnegative selfadjoint extensions $\wt S_1$ and $\wt S_2$ of $S$ their arithmetic mean $\half(\wt S_1\dot+S_2)$, harmonic mean $h(\wt S_1,\wt S_2)$, and nonnegative selfadjoint l.r. $c_0(\wt S_1,\wt S_2)$, given by \eqref{chtlytt},
are nonnegative selfadjoint extensions of $S$;
\item if $\wt Z_1$ and $\wt Z_2$ are parameters of $\mathfrak C(\wt S_1)$ and $\mathfrak C(\wt S_2)$ in formulae \eqref{krparm}, respectively, then the corresponding parameters
for $\mathfrak C\left(\half(\wt S_1\dot+S_2)\right)$, $\mathfrak C\left(h(\wt S_1,S_2)\right)$ and $\mathfrak C(c_0(\wt S_1,\wt S_2))$ are the operators
\[
h(I_{\sN_0}+\wt Z_1,I_{\sN_0}+\wt Z_2)-I_{\sN_0},\;
I_{\sN_0}-h(I_{\sN_0}-\wt Z_1,I_{\sN_0}-\wt Z_2),\quad\mbox{and}\quad \half(\wt Z_1+\wt Z_2),
\]
respectively;
\item
for an arbitrary nonnegative selfadjoint extension $\wt S$ of $S$ the equalities
\begin{equation}\label{osobrol}
\half(S_{\rm F}\dot+\wt S)=S_{\rm F},\;h(S_{\rm K},\wt S)=S_{\rm K}
\end{equation}
hold;
\item
if $\wt S_1$ and $\wt S_2$ are extremal extensions of $S$, then $\half(\wt S_1\dot+S_2)$ and $h(\wt S_1,\wt S_2)$ are extremal extensions of $S$ as well, moreover, for an arbitrary extremal extension $\wt S$ one has
\begin{equation}\label{osobrol2}
\half(S_{\rm K}\dot +\wt S)=\wt S,\;h(S_{\rm F},\wt S)=\wt S.
\end{equation}

\end{enumerate}
\end{theorem}
\begin{proof}
If $\wt S$ is a nonnegative selfadjoint extension of $S$, then (see \cite{Ar4, CS, HSSW06, HSSW07})
\begin{itemize}
\item $S_{\rm K}\le \wt S\le S_{\rm F}$, the form $S_{\rm F}[\cdot,\cdot]$ is the closed restriction of the form $\wt S[\cdot,\cdot],$ and $\cD[\wt S]\cap \cD[S _{\rm F}]=\cD[S_{\rm F}]$.
 \item
 $\wt S^{-1}$ is a nonnegative selfadjoint extension of $S^{-1}$, $S_{\rm F}^{-1}\le \wt S^{-1}\le S_{\rm K}^{-1}$, the form $S_{\rm K}^{-1}[\cdot,\cdot]$ is the closed restriction of the form $\wt S^{-1}[\cdot,\cdot]$, and $\cD[\wt S^{-1}]\cap \cD[S_{\rm K}^{-1}]=\cD[S_{\rm K}^{-1}]$.
\end{itemize}

Let $\wt S_1$ and $\wt S_2$ be two nonnegative selfadjoint extensions of $S$.
Then
\[
\cD[\wt S_1]\cap\cD[\wt S_2]\supseteq\cD[S_{\rm F}],\;\cD[\wt S^{-1}_1]\cap \cD[\wt S^{-1}_2]\supseteq \cD[S_{\rm K}^{-1}],
\]
and
\[
\begin{array}{l}
\half(\wt S_1\dot+\wt S_2)[f]=\half\left(\wt S_1[f]+\wt S_2[f]\right)\ge S_{\rm K}[f],\; f\in \cD[\wt S_1]\cap \cD[\wt S_2],\\
\half(\wt S_1\dot+\wt S_2)[\f]= S_{\rm F}[\f],\; \f\in\cD[S_{\rm F}],\\[3mm]
\half(\wt S^{-1}_1\dot+\wt S^{-1}_2)[g]=\half\left(\wt S^{-1}_1[g]+\wt S^{-1}_2[g]\right)\ge S^{-1}_{\rm F}[g],\; g\in \cD[\wt S^{-1}_1]\cap \cD[\wt S^{-1}_2],\\
\half(\wt S^{-1}_1\dot+\wt S^{-1}_2)[\psi]=S^{-1}_{\rm K}[\psi],\;\psi\in\cD[S^{-1}_{\rm K}].
\end{array}
\]
Hence $\half(\wt S_1\dot+\wt S_2)$ is a nonnegative selfadjoint extension of $S$ and $\half(\wt S^{-1}_1\dot+\wt S^{-1}_2)$ is a nonnegative selfadjoint extension of $S^{-1}.$  It follows that
\[
h(\wt S_1,\wt S_2)=2(\wt S^{-1}_1\dot+\wt S^{-1}_2)^{-1}
\]
is a nonnegative selfadjoint extension of $S$. Besides, we get equalities in \eqref{osobrol}.

Set $Q=\mathfrak C(S),$ $Q_\mu=\mathfrak C(S_{\rm F}),$ $Q_M=\mathfrak C(S_{\rm K}),$ $\sN=\cH\ominus\dom Q,$ $\sN_0=\cran(Q_M-Q_\mu)$.

Let $\wt Q_1:=\mathfrak C(\wt S_1)$ and $\wt Q_2:=\mathfrak C(\wt S_2).$
The operators $\wt Q_k$, $k=1,2$ admit the representations by means of expression \eqref{krparm}:
\[
\wt Q_k=\half(Q_\mu+Q_M)+\half(Q_M-Q_\mu)^\half \wt Z_k(Q_M-Q_\mu)^{\half},
\]
where $\wt Z_k$, $k=1,2$ are selfadjoint contractions in $\sN_0$.
Since the operator $\half(\wt Z_1+\wt Z_2)$ is a selfadjoint contraction and
\[
\half (\wt Q_1+\wt Q_2)=\half(Q_\mu+Q_M)+\half(Q_M-Q_\mu)^\half\left(\half (\wt Z_1+\wt Z_2)\right)(Q_M-Q_\mu)^{\half},
\]
the operator $\half (\wt Q_1+\wt Q_2)$ is a selfadjoint contractive extension of $Q$. Hence $c_0(\wt S_1,\wt S_2)=\mathfrak C(\half (\wt Q_1+\wt Q_2))$ is a nonnegative selfadjoint extension of $S$.

From \eqref{EQ11} and \eqref{EQ22} we get the equalities
\[
I+\mathfrak C\left(\half(\wt S_1\dot+\wt S_2)\right)=h(I+\wt Q_1, I+\wt Q_2),\;
I-\mathfrak C\left(h(\wt S_1,\wt S_2)\right)=h(I-\wt Q_1, I-\wt Q_2).
\]
Further, using \eqref{totjlyj}, one obtains
\[
\left(I+\mathfrak C\left(\half(\wt S_1\dot+\wt S_2)\right)\right)_\sN=\left(h(I+\wt Q_1, I+\wt Q_2)\right)_\sN=h\left((I+\wt Q_1)_\sN, (I+\wt Q_2)_\sN\right).
\]
From \eqref{shoort} and \cite[Proposition 1]{Ar2} it follows the equality
\begin{multline*}
h\left((I+\wt Q_1)_\sN, (I+\wt Q_2)_\sN\right)\\
=2\left(\half(Q_M-Q_\mu)^\half(I_{\sN_0}+\wt Z_1)(Q_M-Q_\mu)^\half\right):\left(\half(Q_M-Q_\mu)^\half(I_{\sN_0}+\wt Z_2)(Q_M-Q_\mu)^\half\right)\\
=\half(Q_M-Q_\mu)^\half h(I_{\sN_0}+\wt Z_1,I_{\sN_0}+\wt Z_2)(Q_M-Q_\mu)^\half.
\end{multline*}
Now from \eqref{krparm}, \eqref{shoort} and the uniqueness of the representation we get that
\[
\mathfrak C\left(\half(\wt S_1\dot+\wt S_2)\right)=\half(Q_\mu+Q_M)+\half(Q_M-Q_\mu)^\half\wt Z(Q_M-Q_\mu)^\half,
\]
where $\wt Z=h(I_{\sN_0}+\wt Z_1,I_{\sN_0}+\wt Z_2)-I_{\sN_0}.$
Similarly
\[
\mathfrak C\left(h(\wt S_1,\wt S_2)\right)=\half(Q_\mu+Q_M)+\half(Q_M-Q_\mu)^\half\wt W(Q_M-Q_\mu)^\half,
\]
where $\wt W=I_{\sN_0}-h(I_{\sN_0}-\wt Z_1,I_{\sN_0}-\wt Z_2).$

Let $\wt S$ be an extremal nonnegative selfadjoint extension of $S$ and let $\wt Q={{\mathfrak C}}(\wt S)$ be its Cayley transform. Then $\wt Q$ admits the representation
\eqref{krparm} with a fundamental symmetry $\wt Z$ in $\sN_0$.
The operator $\wt S^{-1}$ is a nonnegative selfadjoint extension of the operator $S^{-1}$ and
\[
{{\mathfrak C}}(S^{-1})=-Q,\; {{\mathfrak C}}((S^{-1})_{\rm F})=-Q_M,\;{{\mathfrak C}}((S^{-1})_{\rm K})=-Q_\mu,\; {\mathfrak C}(\wt S^{-1})=-\wt Q.
\]
Hence
\[
-\wt Q=-\half(Q_\mu+Q_M)+\half(Q_M-Q_\mu)^\half (-\wt Z)(Q_M-Q_\mu)^{\half}.
\]
Because $\wt Z$ is a selfadjoint and unitary (in $\sN_0$), the operator $-\wt Z$ is  selfadjoint and unitary as well.
Therefore, the operator $\wt S^{-1}$ is an extremal extension of $S^{-1}$.

Let $\wt S_1$ and $\wt S_2$ be two extremal nonnegative selfadjoint extensions of $S$. Then the closed form $\wt S_1[\cdot,\cdot]$ and $\wt S_2[\cdot,\cdot]$ are closed restrictions of the closed form $S_{\rm K}[\cdot,\cdot].$ Therefore the sesquilinear form
$\half(\wt S_1\dot+\wt S_2)[\cdot,\cdot]$ is a closed restriction of the closed form $S_{\rm K}[\cdot,\cdot]$. It follows that $\half(\wt S_1\dot+\wt S_2)$ is an extremal extension of $S$.

Since
$\wt S_1^{-1}$ and $\wt S_2^{-1}$ are extremal nonnegative selfadjoint extensions of $S^{-1}$,  $\half\left(\wt S^{-1}_1\dot+\wt S^{-1}_2\right)$ is an extremal extension of $S^{-1}$
 and this implies that
\[
h(\wt S_1,\wt S_2)=\left(\half\left(\wt S_1^{-1}\dot+\wt S^{-1}_2\right)\right)^{-1}
\]
is an extremal extension of $S$.

For an arbitrary extremal extension $\wt S$ of $S$ we have equalities
\[
\half(S_{\rm K}\dot+\wt S)[f]=\wt S[f]\;\; \forall f\in \cD[(S_{\rm K}\dot+\wt S)/2]=\cD[\wt S],
\]
and
\[
\half(S_{\rm F}^{-1}\dot+\wt S^{-1})[h]=\wt S^{-1}[h]\;\; \forall h\in \cD[(S^{-1}_{\rm F}\dot+\wt S^{-1})/2]=\cD[\wt S^{-1}].
\]
It follows that equalities in \eqref{osobrol2} are valid.

The proof is complete.
\end{proof}

\begin{proposition}\label{extrah}
Let $S$ be a nonnegative symmetric l.r. in $\cH$, having a non-unique nonnegative selfadjoint extension.
Then
 the arithmetic-harmonic means $ah(\wt S_1, \wt S_2 )$ of an arbitrary two extremal nonnegative selfadjoint extensions $\wt S_1$, $\wt S_2$ are extremal nonnegative selfadjoint extensions too. Moreover,
 $$ah(\wt S_1, \wt S_2 )=\left<\wt S_1^{(1)},\wt S_2^{(1)}\right>,$$
  where $\wt S_1^{(1)}=\half(\wt S_1\dot+ \wt S_2 )$,
 $\wt S_2^{(1)}=h(\wt S_1,\wt S_2 )$. In particular, $ah (S_{\rm F}, S_{\rm K})=\left<S_{\rm F},S_{\rm K}\right>$.
\end{proposition}
\begin{proof}
Set $A=S_{\rm F}$, $B=S_{\rm K}$.
 From Theorem \ref{ext} it follows that
\[
h(A,B)=h(S_{\rm F}, S_{\rm K})=S_{\rm K}=B,\; \half(A\dot+B)=\half(S_{\rm F}\dot+ S_{\rm K})=S_{\rm F}=A.
\]
 Then from \eqref{fhbauf1} and \eqref{harar1} we get $A_\infty=S_{\rm F}$ and $B_\infty=S_{\rm K}.$

Recall \cite[Theorem 4.3]{FW} that if $\cL_1$, $\cL_2$ are two subspaces in the Hilbert space $\cH$, then
\begin{equation}\label{proinht}
2(P_{\cL_1}:P_{\cL_2})=P_{\cL_1\cap\cL_2}.
\end{equation}
Let $\wt S_k$, $k=1,2$ be two extremal extensions of $S$. Then the corresponding operators $\wt Z_k$, $k=1,2$ in \eqref{krparm} for $\mathfrak C(\wt S_k)$ are fundamental symmetries in $\sN_0$. Therefore,  the operators $\half (I_{\sN_0}\pm \wt Z_k)$ are orthogonal projections in $\sN_0$.  Hence, the operators
\[
2\left(\half (I_{\sN_0}\pm \wt Z_1)\right):\left(\half(I_{\sN_0}\pm \wt Z_2)\right)=\half h(I_{\sN_0}\pm \wt Z_1,I_{\sN_0}\pm \wt Z_2)
\]
are orthogonal projections in $\sN_0$. It follows that the operators
\[
\wt Z_1^{(1)}=h(I_{\sN_0}+\wt Z_1,I_{\sN_0}+\wt Z_2)-I_{\sN_0},\; \wt Z_2^{(1)}=I_{\sN_0}-h(I_{\sN_0}-\wt Z_1,I_{\sN_0}-\wt Z_2)
\]
are fundamental symmetries in $\sN_0$.
Define two sequences
\[
\wt S_1^{(0)}:=\wt S_1,\;\wt S_2^{(0)}:=\wt S_2,\; \wt S_1^{(n)}=\half(\wt S_1^{(n-1)}\dot+\wt S_2^{(n-1)}),\;
\wt S_2^{(n)}=h(\wt S_1^{(n-1)}, \wt S_2^{(n-1)}),\; n\in\dN.
\]
Let us show that
$$\wt S_1^{(n)}=\wt S_1^{(1)},\;\wt S_2^{(n)}=\wt S_2^{(1)}\;n\ge 2.$$
Define also two sequences of operators in $\sN_0$:
\[
\begin{array}{l}
\wt Z_1^{(n)}=h(I_{\sN_0}+\wt Z_1^{(n-1)},I_{\sN_0}+\wt Z_2^{(n-1)})-I_{\sN_0},\\
 \wt Z_2^{(n)}=I_{\sN_0}-h(I_{\sN_0}-\wt Z_1^{(n-1)},I_{\sN_0}-\wt Z_2^{(n-1)}),\; n\ge 2.
 \end{array}
\]
By Theorem \ref{ext} the operators $\wt Z_1^{(n)}$ and $\wt Z_2^{(n)}$ correspond  for all $n\in\dN$ to $\wt Q_1^{(n)}=\mathfrak C(\wt S_1^{(n)})$, $\wt Q_2^{(n)}=\mathfrak C(\wt S_2^{(n)})$ in their representation in \eqref{krparm}, respectively.

Because $\wt S_2^{(1)}\le \wt S_1^{(1)}$, we get $\wt Z_2^{(1)}\ge \wt Z_1^{(1)}$. Hence
$$I_{\sN_0}+\wt Z_2^{(1)}\ge I_{\sN_0}+ \wt Z_1^{(1)},\;I_{\sN_0}-\wt Z_1^{(1)}\ge I_{\sN_0}- \wt Z_2^{(1)}.$$
Since the operators $\half(I_{\sN_0}\pm\wt Z_1^{(1)})$, $\half(I_{\sN_0}\pm\wt Z_2^{(1)})$ are orthogonal projections in $\sN_0$,
the equalities
\[
I_{\sN_0}+\wt Z_1^{(2)}=h(I_{\sN_0}+ \wt Z_1^{(1)},I_{\sN_0} + \wt Z_2^{(1)}),\;I_{\sN_0}-\wt Z_2^{(2)}=h(I_{\sN_0}- \wt Z_1^{(1)},I_{\sN_0}- \wt Z_2^{(1)})
\]
and equality \eqref{proinht}
imply that
\[
I_{\sN_0}+\wt Z_1^{(2)}=I_{\sN_0}+\wt Z_1^{(1)},\; I_{\sN_0}-\wt Z_2^{(2)}=I_{\sN_0}-\wt Z_2^{(1)}.
\]
Therefore $\wt Z_1^{(2)}=\wt Z_1^{(1)}$, $\wt Z_2^{(2)}=\wt Z_2^{(1)}$, and by induction
\[
\wt Z_1^{(n)}=\wt Z_1^{(1)},\;\wt Z_2^{(n)}=\wt Z_2^{(1)},\; \wt S_1^{(n)}=\wt S_1^{(1)},\;\wt S_2^{(n)}=\wt S_2^{(1)}\;\forall n\ge 2.
\]
Thus,
$
ah(\wt S_1,\wt S_2)=\left<\wt S_1^{(1)},\wt S_2^{(1)}\right>.
$
By Theorem \ref{ext} the operators $\wt S_1^{(1)},$ $\wt S_2^{(1)}$ are extremal nonnegative selfadjoint extensions of $S$.
\end{proof}
\begin{proposition}\label{jcnfy}
Consider in the Hilbert space $\cL_2(\dR_+)$ the following differential operators
\[
L_cf=-\cfrac{d^2f}{d x^2},\;\left\{\begin{array}{l}\dom L_c=\left\{f\in W_2^2(\dR_+),\; f'(0)=c f(0)\right\},\; c>0,\; c\ne \infty,\\[2mm]
\dom L_0=\left\{f\in W_2^2(\dR_+),\; f'(0)=0\right\},\\[2mm]
\dom L_\infty=\left\{f\in W_2^2(\dR_+),\; f(0)=0\right\},
\end{array}\right.
\]
where $W_2^2(\dR_+)$ is the Sobolev space.
Then
\[\begin{array}{l}
\half(L_c\dot+L_d)=L_{\half(c+d)},\;c_0(L_c,L_d)=L_{c_0(c,d)},\; h(L_c, L_d)=L_{h(c,d)}=L_{\frac{2 c d}{c+d}}, \;ah(L_c,L_d)=L_{\sqrt{cd}},\\[2mm]
\half(L_c\dot+L_\infty)=L_{\infty},\;  \;c_0(L_c,L_\infty)=L_{2c+1},\;
 h(L_c, L_\infty)=L_{2c},\;ah(L_c,L_\infty)=L_{\infty},\; c> 0\\[2mm]
\half(L_c\dot+L_0)=L_{c/2},\;c_0(L_c, L_0)=L_{\frac{c}{2+c}},\;
 h(L_c, L_0)=L_{0}, \;ah(L_c,L_0)=L_{0} ,\;c> 0\\[2mm]
\half(L_0\dot+L_\infty)=L_{\infty},\; c_0(L_0, L_\infty)=L_1,\;
 h(L_0, L_\infty)=L_{0}, \;ah(L_\infty,L_0)=\left<L_{\infty}, L_0\right>.
  \end{array}
\]
\end{proposition}
\begin{proof}
The operators $L_c$ are nonnegative selfadjoint  extensions of the nonnegative symmetric operator
\[
Sf=-\cfrac{d^2f}{d x^2},\;\dom S=\left\{f\in W_2^2(\dR_+),\; f(0)=f'(0)=0\right\}.
\]
The operator $S$ has one-dimensional deficiency subspaces. In particular,
\[
\sN=\{\xi \exp(-x),\; \xi\in\dC\}
\]
is the deficiency subspace corresponding to $\lambda=-1$.

One can show that
\[
(L_c+I)^{-1}\exp(-x)=\half\exp(-x)\left(x+\cfrac{1}{c+1}\right),\; x\in\dR_+.
\]
Then
\begin{equation}\label{hfpyjcn}
\begin{array}{l}
\left((L_c+I)^{-1}-(L_\infty+I)^{-1}\right)\exp(-x)=\cfrac{\exp(-x)}{2(c+1)},\\[3mm]
\left((L_0+I)^{-1}-(L_\infty+I)^{-1}\right)\exp(-x)=\frac{1}{2}\exp(-x).
\end{array}
\end{equation}
The operators $L_\infty$ and $L_0$ are the Friedrichs and the Kre\u{\i}n extension of $S$, respectively. Hence as it is shown above
\[
\half(L_0\dot+L_\infty)=L_{\infty},\;
 h(L_0, L_\infty)=L_{0}, \;ah(L_\infty,L_0)=\left<L_{\infty}, L_0\right>.
\]
Set
\[
\wt Q_c=\mathfrak C(L_c),\;  Q_\mu=\mathfrak C(L_\infty),\; Q_M=\mathfrak C(L_0).
\]
From \eqref{rtkbtr} we have
\[
\wt Q_c-Q_\mu=2\left((L_c+I)^{-1}-(L_\infty+I)^{-1}\right),\;Q_M-Q_\mu=2\left((L_0+I)^{-1}-(L_\infty+I)^{-1}\right)
\]
and $\ran (\wt Q_c-Q_\mu)=\ran(Q_M-Q_\mu)=\sN$.
Now from \eqref{krparm} it follows that
\[
2\left((L_c+I)^{-1}-(L_\infty+I)^{-1}\right)=\left(1+\wt z_c\right)\left((L_0+I)^{-1}-(L_\infty+I)^{-1}\right).
\]
Using \eqref{hfpyjcn} we obtain
$$\wt z_c=\cfrac{1-c}{1+c},\; z_0=1, \; z_\infty=-1.$$
Hence
\[
h(1+\wt z_d,1+\wt z_d)-1=\cfrac{1-\half(c+d)}{1+\half(c+d)}, \;1-h(1-\wt z_d,1-\wt z_d)=\cfrac{1-\cfrac{2cd}{c+d}}{1+\cfrac{2cd}{c+d}}
\]
Applying Theorem \ref{ext} and \eqref{chtlytt}, \eqref{chtlyttt}, we get
\[
\begin{array}{l}
\half(L_c\dot+L_d)=L_{\half(c+d)},\; c_0(L_c, L_d)=L_{c_0(c,d)},\; h(L_c, L_d)=L_{h(c,d)}=L_{\frac{2 c d}{c+d}},\\
\half(L_c\dot+L_\infty)=L_{\infty},\;c_0(L_c, L_\infty)=L_{2c+1},\;h(L_c, L_\infty)=L_{2c},\\
\half(L_c\dot+L_0)=L_{c/2},\;c_0(L_c L_0)=L_{\frac{c}{2+c}},\; h(L_c, L_0)=L_{0}.
\end{array}
\]
Set
\[
A_1=\half(L_c\dot+L_d),\; B_1=h(L_c, L_d),\ldots, A_n=\half(A_{n-1}\dot+B_{n-1}),\;B_n=h(A_{n-1}, B_{n-1}),\; n\ge 2,
\]
\[\begin{array}{l}
c_1=\half(c+d),\; d_1=h(c,d)=\cfrac{2cd}{c+d},\ldots, \\
c_n=\half(c_{n-1}+d_{n-1}),\; d_n=h(c_{n-1},d_{n-1})=\cfrac{2c_{n-1}d_{n-1}}{c_{n-1}+d_{n-1}},\; n\ge 2.
\end{array}
\]
\[
\wt z_{c_n}=\frac{1-c_n}{1+c_n},\; \wt z_{d_n}=\frac{1-d_n}{1+d_n},\; n\in\dN.
\]
Then
\[
A_n=L_{c_n},\;B_n=L_{d_n}\;\forall n\in\dN.
\]
Due to Theorem \ref{ext} we have the equalities for the Cayley transforms
\[
\begin{array}{l}
\mathfrak C(A_n)=\half(Q_\mu+Q_M)+\half \wt z_{c_n}(Q_M-Q_\mu),\\[3mm]
\mathfrak C(B_n)=\half(Q_\mu+Q_M)+\half \wt z_{d_n}(Q_M-Q_\mu),\; n\in\dN.
\end{array}
\]
From
\[
\lim\limits_{n\to\infty} c_n=\lim\limits_{n\to\infty}d_n=\sqrt{cd},
\]
it follows that
\[
{\rm s}-\lim\limits_{n\to \infty}\mathfrak C(A_n)={\rm s}-\lim\limits_{n\to \infty}\mathfrak C(B_n)=\half(Q_\mu+Q_M)+\frac{1-\sqrt{cd}}{1+\sqrt{cd}}\,\cfrac{Q_M-Q_\mu}{2}.
\]
Therefore
\[
{\rm s-R}-\lim\limits_{n\to\infty}L_{c_n}={\rm s-R}-\lim\limits_{n\to\infty}L_{d_n}=L_{\sqrt{cd}}.
\]
Thus, $ah(L_c, L_d)=L_{\sqrt{cd}}$ for the case $\left\{c>0,d\in[0,+\infty]\right\}$.
\end{proof}
Observe that the equality
$$\dim \ran \left((S_{\rm K}+I)^{-1}-(S_{\rm F}+I)^{-1}\right)=\dim \left(\cD[S_{\rm K}]/\cD[S_{\rm F}]\right)$$
takes place \cite{Kr}. Besides if
\begin{equation}\label{hfvthy}
\dim \ran \left((S_{\rm K}+I)^{-1}-(S_{\rm F}+I)^{-1}\right)=1,
\end{equation}
then for an arbitrary nonnegative selfadjoint extension $\wt S$ of $S$ from \eqref{rtkbtr} and \eqref{krparm} it follows that
\begin{equation}\label{resajh}
\left(\wt S+I\right)^{-1}=\half\left((S_{\rm K}+I)^{-1}+(S_{\rm F}+I)^{-1}\right)+\cfrac{\wt z}{2}\left((S_{\rm K}+I)^{-1}-(S_{\rm F}+I)^{-1}\right),
\end{equation}
for some $\wt z\in[-1,1].$
\begin{theorem}\label{yfltdct}
Let $S$ be a nonnegative symmetric l.r. such that \eqref{hfvthy} holds.
Suppose that $\wt S_1$ and $\wt S_2$ are two different non-extremal nonnegative selfadjoint extensions of $S.$
Then the arithmetic-harmonic mean $ah(\wt S_1,\wt S_2)$ is singleton. Moreover,
if the numbers $\wt z_1, \wt z_2\in [-1,1]$ correspond to $\wt S_1$ and $\wt S_2$ in the resolvent formula \eqref{resajh}, then
\begin{multline}\label{utjvnh}
\left(ah(\wt S_1,\wt S_2)+I\right)^{-1}=\half\left((S_{\rm K}+I)^{-1}+(S_{\rm F}+I)^{-1}\right)\\
+\cfrac{\wt w}{2}\left((S_{\rm K}+I)^{-1}-(S_{\rm F}+I)^{-1}\right),
\end{multline}
where
\begin{equation}\label{ghtl11}
\wt w=\cfrac{\sqrt{(1+\wt z_1)(1+\wt z_2)}-\sqrt{(1-\wt z_1)(1-\wt z_2)}}{\sqrt{(1+\wt z_1)(1+\wt z_2)}+\sqrt{(1-\wt z_1)(1-\wt z_2)}}.
\end{equation}
In particular,
$ah(S_{\rm F},\wt S)=S_{\rm F}$, when $\wt S\ne S_{\rm K}$ and $ah(S_{\rm K},\wt S)=S_{\rm K}$, when $\wt S\ne S_{\rm F}.$
\begin{proof} Clearly, $\left<\wt S_1,\wt S_2\right>\ne \left<\wt S_{\rm F},\wt S_{\rm K}\right>.$
Define the iterations
\[
\wt S_1^{(0)}:=\wt S_1,\;\wt S_2^{(0)}:=\wt S_2,\; \wt S_1^{(n)}=\half(\wt S_1^{(n-1)}\dot+\wt S_2^{(n-1)}),\;
\wt S_2^{(n)}=h(\wt S_1^{(n-1)}, S_2^{(n-1)}),\; n\in\dN.
\]
Then from \eqref{resajh} and Theorem \ref{ext}
\[
\left(\wt S_j^{(n)}+I\right)^{-1}=\half\left((S_{\rm K}+I)^{-1}+(S_{\rm F}+I)^{-1}\right)+\cfrac{ \wt z_j^{(n)}}{2}\left((S_{\rm K}+I)^{-1}-(S_{\rm F}+I)^{-1}\right),\;
j=1,2,\; n\in\dN,
\]
where $\wt z_1^{(0)}=\wt z_1,$ $\wt z_2^{(0)}=\wt z_2$,
\begin{equation}\label{rekur}
\left\{\begin{array}{l}1+\wt z_1^{(n)}=h(1+\wt z_1^{(n-1)},1+\wt z_2^{(n-1)})=\cfrac{2(1+\wt z_1^{(n-1)})(1+\wt z_2^{(n-1)})}{2+\wt z_1^{(n-1)}+\wt z_2^{(n-1)}},\\
1-\wt z_2^{(n)}=h(1-\wt z_1^{(n-1)},1-\wt z_2^{(n-1)})=\cfrac{2(1-\wt z_1^{(n-1)})(1-\wt z_2^{(n-1)})}{2-\wt z_1^{(n-1)}-\wt z_2^{(n-1)}}.
\end{array}
\right.
\end{equation}
Besides, $\{\wt z_1^{(n)}\}$ and  $\{\wt z_2^{(n)}\}$ are  non-decreasing and non-increasing sequences, respectively, and $\wt z_1^{(n)}\le \wt z_2^{(n)}$ $\forall n\in\dN$.
Hence the sequences $\{\wt z_1^{(n)}\}$ and  $\{\wt z_2^{(n)}\}$ converge. Set
\[
\wt w_j=\lim\limits_{n\to \infty}\wt z_j^{(n)}, \; j=1,2.
\]
Note that because $\left<\wt S_1,\wt S_2\right>\ne \left<\wt S_{\rm F},\wt S_{\rm K}\right>$ we have $\wt z_1\wt z_2\ne -1$.

If $\wt z_1=-1$, $\wt z_2\ne 1$, then from \eqref{rekur} one has $\wt z_1^{(n)}=-1$ for all $n\in\dN$ and therefore $\wt w_1=\wt w_2=-1$.
If $ \wt z_2= 1$, $\wt z_1\ne-1$, then $\wt w_1=\wt w_2=1$.

Suppose $\wt z_1\ne -1,$ $\wt z_2\ne -1$. Then \eqref{rekur} yields $\wt w_1=\wt w_2:=\wt w$. Let us find $\wt w$.
Set
\[
a_n:=\mathfrak C(\wt z_1^{(n)})=\cfrac{1-\wt z_1^{(n)}}{1+\wt z_1^{(n)}},\; b_n:=\mathfrak C(\wt z_2^{(n)})=\cfrac{1-\wt z_2^{(n)}}{1+\wt z_n^{(n)}},\; n\in\dN_0.
\]
Then, using \eqref{EQ11} and \eqref{EQ22}, we obtain
\[
a_n=\half(a_{n-1}+b_{n-1}),\; b_n=h(a_{n-1}, b_{n-1})=\cfrac{2a_{n-1} b_{n-1}}{a_{n-1}+ b_{n-1}},\; n\in\dN_0.
\]
Hence,
\[
\lim\limits_{n\to\infty} a_n=\lim\limits_{n\to\infty} b_n=\sqrt{a_0 b_0}=\sqrt{\cfrac{1-\wt z_1}{1+\wt z_1}\,\cfrac{1-\wt z_2}{1+\wt z_2}}.
\]
Thus, we arrive at \eqref{utjvnh}--\eqref{ghtl11}.
\end{proof}

\end{theorem}


\begin{thebibliography}{AHS}

\bibitem{And}
W.N.~Anderson, Shorted operators, SIAM J. Appl. Math. \textbf{20} (1971), 520--525.

\bibitem{AD}
W.N.~Anderson and R.J.~Duffin,
Series and parallel addition of matrices,
J. Math. Anal. Appl. \textbf{26} (1969), 576--594.

\bibitem{AT}
W.N.~Anderson and G.E.~Trapp, Shorted operators, II, SIAM J. Appl. Math. \textbf{28} (1975), 60--71.

\bibitem{AMT}
W.N.~Anderson, T.D.~Morlye, and G.E.~Trapp, Characterizations of parallel subtraction, Proc. Nat. Acad. Sci. USA, \textbf{76} (1975), No.8, 3599--3601.


\bibitem{Ando1976}
T.~Ando, Lebesgue type decomposition of positive operators. Acta Sci.Math. (Szeged)
\textbf{38} (1976), 253--260.

\bibitem{Ando1978}
T.~Ando, Topics on operator inequalities, Lecture Notes, Hokkaido University, Sapporo, Japan, 1978.

\bibitem{AN}
T.~Ando and K.~Nishio, Positive self-adjoint extensions of positive
symmetric Operators. Toh\'oku Math. J. \textbf{22} (1970), 65--75.

\bibitem{ACS}
J.~Antezana, G.~Corach, D.~Stojanoff, Bilateral shorted operators and parallel sums, Linear Algebra and Application \textbf{414} (2006), 570--588.

\bibitem{Ar}
R.~Arens, Operational calculus of linear relations, Pacific J.
Math. \textbf{11} (1961), 9--23.

\bibitem{Ar2}
Yu.M.~Arlinski\u{\i}, On the theory of operator means. Ukr.
Mat. Zh. \textbf{42} (1990), No.6, 723--730 (in Russian). English
translation in Ukr. Math. Journ. \textbf{42} (1990), No.6,
639--645.


\bibitem{Ar4}
Yu.M.~Arlinski\u{\i}, Extremal extensions of sectorial linear
relations. Matematychnii Studii \textbf{7} (1997), No.1,  81--96.

\bibitem{ArlBelTsek2011}
Yu.~Arlinski\u{\i}, S.~Belyi, and E.~Tsekanovski\u{\i}, Conservative
realizations of Herglotz-Nevanlinna functions, Operator Theory:
Advances and Applications \textbf{217}. Basel: Birkhauser, xviii, 528
p.(2011).

\bibitem{ArlHassi_2019}
Yu.~Arlinski\u{\i} and S.~Hassi, Holomorphic operator valued functions generated by passive
selfadjoint systems, in Interpolation and Realization Theory with Applications to Control Theory, Operator Theory: Advances and Appl. \textbf{272} (2019), 35--76.


\bibitem{AHZS}
Yu.M.~Arlinski\u{\i}, S.~Hassi, Z.~Sebesty\'en, and H.S.V.~de~Snoo,
 On the class of extremal extensions of a nonnegative
operator, Oper. Theory: Adv. Appl. \textbf{127}, 41--81  (2001).

\bibitem{AG}
Gr.~Arsene and A.~Gheondea, Completing matrix contractions, J.
Operator Theory \textbf{7} (1982), 179--189.

\bibitem{ADW}
T.Ya.~Azizov, A.~Dijksma, and G.~Wanjala, Compressions of maximal dissipative and self-adjoint linear relations and of dilations. Linear Algebra Appl. \textbf{439} (2013), 771--792.

\bibitem{BHSW2010}
Ju.~Behrndt, S.~Hassi, H.~de Snoo, and R.~Wietsma, Monotone
convergence theorems for semi-bounded operators and forms with
applications, Proc. Royal Soc. of Edinburgh \textbf{140} A (2010), 927--951.

\bibitem{CS}
E.A.~Coddington and H.S.V.~de Snoo, Positive selfadjoint extensions
of positive symmetric subspaces, Math. Z. \textbf{159} (1978), 203--214.

\bibitem{CMM}
M.~Contino, A.~Maestripieri, and S.~Marcantognini,
Schur complements of selfadjoint Kre\u{\i}n space operators, Linear Algebra and its Applications, \textbf{581}  (2019), 214--246.

\bibitem{DaKaWe}
Ch.~Davis, W.M.~Kahan, and H.F.~Weinberger, Norm preserving
dilations and their applications to optimal error bounds, SIAM J.
Numer. Anal. \textbf{19} (1982), 445--469.


\bibitem{DM1991} V.A.~Derkach and M.M.~Malamud, Generalized
resolvents and the boundary value problems for Hermitian operators
with gaps, J. Funct. Anal. \textbf{95} (1991), No. 1, 1--95.


\bibitem {DHMS06}
V.~Derkach, S.~Hassi, M.M.~Malamud, and H.S.V.~de Snoo, Boundary
relations and their Weyl families, Trans. Amer. Math. Soc. \textbf{358}
(2006), No. \textbf{12}, 5351--5400

\bibitem{Doug}
R.G.~Douglas, On majorization, factorization and range
inclusion of operators in Hilbert space, Proc. Amer. Math. Soc.
\textbf{17} (1966), 413--416.


\bibitem{Faris}
W.G.~Faris, Selfadjoint operators, Lecture Notes in Mathematics, \textbf{433} (1975).


\bibitem{FM}
B.~Farkas and M.~Matolsci, Commutation properties of the form sum of positive, symmetric operators, Acta Sci.Math. (Szeged) \textbf{67} (2001), 777--790.

\bibitem{FW}
P.A.~Fillmore and J.P.~Williams, On operator ranges. Advances in
Math. \textbf{7} (1971), 254--281.

\bibitem{FGK}
J.~Fridrich, M.~G\"{u}ntner, and L.~Klotz, A generalized Schur complement for nonnegative operators on linear spaces, Banach J. Math. Anal \textbf{12} (2018), No.3, 617--633.

\bibitem{HMS}
S.~Hassi, M.~Malamud, and H.S.V.~de~Snoo, On Kre\u{\i}n's extension theory of nonnegative operators, Math. Nach. \textbf{274-275} (2004), 40--73.

\bibitem{HSSW06}
S.~Hassi, A.~Sandovici, H.S.V.~de~Snoo, and H.~Winkler, Form sums of nonnegative selfadjoint operators, Acta Math. Hungar. \textbf{111} (1-2) (2006), 81--105.


\bibitem{HSSW07}
S.~Hassi, A.~Sandovici, H.S.V.~de~Snoo, and H.~Winkler, A general
factorization approach to the extension theory of nonnegative
operators and  relations, J. Operator Theory \textbf{58} (2007), 351--386.


\bibitem{HSdeS2009}
S.~Hassi, Z.~Sebestyen, and H.S.V.~de~Snoo, Lebesgue type decompositions for nonnegative forms,
J. Funct. Anal. \textbf{257} (2009), No.12, 3858--3894.

\bibitem{HSdeS2010}
 S.~Hassi, Z.~Sebestyen, and H.S.V.~de~Snoo, Domain and range descriptions for adjoint relations, and parallel sums and differences of forms. Recent advances in operator theory in Hilbert and Krein spaces, 211--227, Oper. Theory Adv. Appl. \textbf{198}, Birkhauser Verlag, Basel, 2010.

\bibitem{Ka}
T.~Kato, Perturbation theory for linear operators, Springer-Verlag,
Berlin, Heidelberg, 1995.


\bibitem{Kr}
 M.G.~Kre\u{\i}n, Theory of selfadjoint extensions of semibounded operators
and its applications. I,  Mat. Sb. \textbf{20} (1947), No.3, 431--498.
(in Russian)

\bibitem{KrO}
M.G.~Kre\u{\i}n and I.E.~Ovcharenko, On $Q$-functions and
sc-extensions of nondensely defined Hermitian contractions, Sibirsk.
Mat. Zh. \textbf{18} (1977), No.5, 1032--1056 (in Russian). English translation
in Siberian Math. Journ. \textbf{18} (1977), No.5, 728--746.

\bibitem{K-A}
F.~Kubo and T.~Ando, Means of positive linear operators. Math. Ann. \textbf{246} (1980), No.3, 205--224.

\bibitem{MP}
A.~Maestripieri and F.M.~Peria, Schur complements in Kre\u{\i}n spaces, Integral Equations Operator Theorey, \textbf{59} (2007), No.2, 207--221.

\bibitem{MBM}
S.K.~Mitra, P.~Bhimasankaram, S.B.~Malik, Matrix partial orders, shorted operators and applications,
Series in Algebra 10. Hackensack, NJ: World Scientific, 2010.

\bibitem{Morley}
T.D.~Morley, Shorts of block operators and infinite networks--A note of the shorted operator: II, Circuits Systems Signal Process \textbf{9} (1990), No.2, 161--170

\bibitem{NA}
K.~Nishio and T.~Ando,
Characterizations of operators derived from network connections, J. Math. Anal. Appl. \textbf{53} (1976), 539--549.

\bibitem{Nud}
M.A.~Nudelman, A generalization of Stenger's lemma to maximal dissipative operators, Integral Equations Operator Theory \textbf{70} (2011), 301--305.

\bibitem{P}
E.L.~Pekarev, Shorts of operators and some extremal problems,
Acta Sci. Math. (Szeged) \textbf{56} (1992), 147--163.

\bibitem{P_2014}

E.L.~Pekarev,
A note on characterizations of the shorted operators, Electronic Journal of Linear Algebra  \textbf{27} (2014), 155--161.


\bibitem{PSh}
E.L.~Pekarev and Yu.L.~Shmul'yan, Parallel addition and
parallel subtraction of operators. Izv. AN SSSR \textbf{40} (1976),
366--387 (1976) (in Russian). English translation in Math. USSR Izv.
\textbf{10}, 289--337.

\bibitem{PuWo}
W.~Pusz and S.L.~Woronowics, Functional calculus for sesquilinear forms and purification map, Rep. Math. Phys., \textbf{8} (1975), No.2, 159--170.

\bibitem{RoBe}
F.S.~Rofe-Beketov, Numerical range of linear relation and maximal
linear relations, Functions Theory, Functional Anal. and their Appl.
\textbf{44} (1985), 103--111  (in Russian). English translation in
Journal of Math. Sci. \textbf{48} (1990), No.3, 329--336.



\bibitem{Shmul}
Yu.L.~Shmul'yan, Hellinger's operator integral, Mat. Sb. \textbf{49} (1959),
No.4, 381--430 (in Russian).

\bibitem{ShYa}
Yu.L.~Shmul'yan and R.N.~Yanovskaya, Blocks of a contractive
operator matrix, Izv. Vuzov, Mat. \textbf{7} (1981), 72-75 (in Russian). English translation in
Sov. Math. \textbf{25} (1981), No. 7, 82--86.

\bibitem{Simon1978}
B.~Simon, A canonical decomposition for quadratic forms with applications to monotone convergence theorems, J. Funct. Analysis \textbf{28} (1978), 377--385.

\bibitem{Stenger}
W.~Stenger, On the projection of a selfadjoint operator, Bull. Amer.Math.Soc. \textbf{74} (1968), 369--372.


\bibitem{Tarcsay2013}
Z.~Tarcsay, Lebesgue type decomposition of positive operators. Positivity \textbf{17} (2013), No.3, 803--817.


\bibitem{Tarcsay2015}
Z.~Tarcsay, On the parallel sum of positive operators, forms and functionals, Acta Math. Hungar. \textbf{147} (2015), No.2, 408--426.

\bibitem{Zhang}
F.~Zhang, The Schur complement and its applications. Springer, 2005.

\end{thebibliography}
\end{document}